\documentclass{article}%
\usepackage{amsmath}
\usepackage{amsfonts}
\usepackage{amssymb}
\usepackage{graphicx}
\usepackage{enumerate}
\usepackage{color}
\usepackage{hyperref}
\usepackage{graphicx}%
\newtheorem{theorem}{Theorem}

\newtheorem{definition}[theorem]{Definition}
\newtheorem{example}[theorem]{Example}

\newtheorem{lemma}[theorem]{Lemma}

\newtheorem{proposition}[theorem]{Proposition}
\newtheorem{remark}[theorem]{Remark}

\newenvironment{proof}[1][Proof]{\noindent\textbf{#1.} }{\ \rule{0.5em}{0.5em}}
\begin{document}

\title{Nonautonomous control systems and skew product flows}
\author{Fritz Colonius\\Institut f\"{u}r Mathematik, Universit\"{a}t Augsburg, \\Universit\"{a}tsstrasse 14, 86159 Augsburg, Germany,\\fritz.colonius@uni-a.de
\and Roberta Fabbri\\Dipartimento di Matematica e Informatica `Ulisse Dini', \\Universit\`{a} degli Studi di Firenze, \\Viale Giovanni Battista Morgagni 67/a, 50134 Firenze, Italy,\\roberta.fabbri@unifi.it}
\maketitle

\begin{abstract}
For nonautonomous control systems with compact control range, associated
control flows are introduced. This leads to several skew product flows with
various base spaces. The controllability and chain controllability properties
are studied and related to properties of the associated skew product flows.

\end{abstract}





\section{Introduction\label{Section1}}

The goal of this paper is to analyze controllability properties of
nonautonomous control-affine systems of the form%
\begin{equation}
\dot{x}(t)=f_{0}(\omega\cdot t,x(t))+\sum_{i=1}^{m}u_{i}(t)f_{i}(\omega\cdot
t,x(t)),\quad u(t)=(u_{i}(t))_{i=1,\dots,m}\in U, \label{control}%
\end{equation}
where $f_{i}:\Omega\times M\rightarrow TM,i=0,1,\dots,m$, are continuous maps,
$\Omega$ is a compact metric space and $M$ denotes a $d$-dimensional connected
smooth ($C^{\infty}$)-manifold with tangent bundle $TM$. For every $\omega
\in\Omega$, the map $f_{i}(\omega,\cdot)$ is assumed to be a  $C^{1}$-vector field on $M$ for any $i=0,\ldots, m$, 
the set $U\subset\mathbb{R}^{m}$ is compact and convex, and the control
functions $u$ are taken  in%
\[
\mathcal{U}:=\left\{u\in L^{\infty}(\mathbb{R},\mathbb{R}^{m})\left\vert
u(t)\in U\text{ for almost all }t\in\mathbb{R}\right.  \right\}  .
\]
Furthermore, $(\Omega,\sigma)$ is a minimal continuous flow with
$\sigma:\mathbb{R}\times\Omega\rightarrow\Omega$ written as $\sigma
(t,\omega)=\sigma_{t}(\omega)=\omega\cdot t,t\in\mathbb{R}$.

We suppose that for every $u\in\mathcal{U}$ and $(\omega,x_{0})\in\Omega\times
M$ there exists a unique (Carath\'{e}odory) solution $x(t)=\varphi
(t,\omega,x_{0},u)$ with $x(0)=x_{0}$ on a maximal open interval in
$\mathbb{R}$. For existence theory, we refer e.g. to Walter \cite[Supplement
II of \S \ 10]{Wal}, Bressan and Piccoli \cite[Chapter 2]{BreP}, Kawan
\cite[Section 1.2]{Kawan13}. Instead of analyzing the behavior of system
\eqref{control} for a single excitation $\omega$, we allow all excitations in
$\Omega$.

The present paper combines two lines of research. A classical approach to
nonautonomous differential equations embeds the equation into a family of
equations depending on a driving system in the background. In particular, this
can be achieved by the so-called hull construction going back to Bebutov
\cite{Be}. This led to the development of skew product dynamical systems
pioneered by Miller \cite{Miller} and Sell \cite{Sell, Sell1}. This concept
mainly motivated by almost periodic equations (cf. Shen and Yi \cite{ShYi},
Zhao \cite{ZH}) has found wide ranging generalizations and applications in
spectral theory for finite dimensional systems (Sacker and Sell \cite{SS78}),
Hamiltonian systems (Johnson, Obaya, Novo, N\'{u}\~{n}ez, Fabbri
\cite{JONNF16}), for infinite dimensional (Smith \cite{Smi}, Shen and Yi
\cite{ShYi}) and random dynamical systems (Arnold \cite{Ar98}). The literature
in these active fields is huge, we only cite Kloeden and Rasmussen
\cite{KloR11}, Carvalho, Langa, and Robinson \cite{CLR}, Novo, N\'{u}\~{n}ez,
and Obaya \cite{NovNO05}), and N\'{u}\~{n}ez, Obaya, and Sanz \cite{NOS1}. We
also mention that, using a pullback approach for chain transitivity, Chen and
Duan \cite{ChenDu11} constructed a state space decomposition for nonautonomous
dynamical systems on a non-compact state space of a skew product flow.

In control theory, an approach to autonomous control systems introduces a
dynamical system, the control flow, on the product of the state space of the
differential equation with an appropriate set of control functions endowed
with the right shift. This again yields a skew product flow and, in
particular, leads to advances in controllability and stability problems by
exploiting concepts and tools from dynamical systems. The theory of control
flows, control sets, and chain control sets is developed in Colonius and
Kliemann \cite{ColK00} and Kawan \cite{Kawan13}. For further contributions we
refer to Ayala, da Silva, and Mamani \cite{ASM23}, da Silva \cite{daS23},
Cavalheiro, Cossich, and Santana \cite{CCS24}, Boarotto and Sigalotti
\cite{BoaS20}, Tao, Huang, and Chen \cite{TaoHC}.

Systems of the form (\ref{control}) include both: they are nonautonomous by
the presence of a driving system on $\Omega$ and additionally controls in
$\mathcal{U}$ are present. There are various ways to look at them: they are
nonautonomous control systems in $M$ with states $x\in M$; autonomous control
systems in $\Omega\times M$ with extended states $(\omega,x)\in\Omega\times
M$; and they are dynamical systems, nonautonomous control flows in
$\mathcal{U}\times\Omega\times M$ with states $(u,\omega,x)\in\mathcal{U}%
\times\Omega\times M$. In Section \ref{sec:preli} we will discuss
corresponding skew product flows.

We will introduce nonautonomous control sets and chain control sets (defined
by generalized controllability properties) as generalizations of the
corresponding autonomous versions. The main results of the present paper are
Theorem \ref{Theorem_single_fiber}, which analyzes when chain control sets are
determined by a single fiber for the corresponding control flow; Theorem
\ref{Theorem_equivalence}\textbf{ }showing that the chain control sets
uniquely correspond to the (appropriately defined) maximal chain transitive
sets of the control flow; Theorem \ref{Theorem_exact} shows that nonautonomous
equilibria for the uncontrolled system (i. e., for $u(t)\equiv0$) are in
control sets, and Theorem \ref{Theorem_equivalence2} relates these control
sets to topologically mixing sets of the control flow. Nonautonomous
equilibria have been studied, in particular, for monotone flows.

In the rest of this introduction, we describe the structure of the paper,
which develops along five sections. Section \ref{sec:preli} is a preliminary
section where basic notions and results for nonautonomous dynamical and
control systems are presented together with an analysis of the control flow in
$\mathcal{U}\times\Omega\times M$ of the system \eqref{control} as a
continuous dynamical system. A scalar example, Example \ref{Example_hull}, is
presented due to Elia, Fabbri, and N\'{u}\~{n}ez \cite{EFN25}. This uses the
so-called hull construction. Section \ref{transitivity} relates chain control
sets to invariant chain transitive sets of the nonautonomous control flow. The
final Section \ref{NONAUTONOMOUS-CS} defines nonautonomous control sets for
system \eqref{control}. Controllability properties are analyzed in a
neighborhood of a nonautonomous equilibrium.

The paper completes and generalizes some of the results obtained by Colonius and
Wichtrey \cite{ColW09} for control systems described by ordinary differential
equations subject to almost periodic excitations. With the time-translation,
these excitations generate a minimal ergodic flow on a compact metric space.

The bifurcation results for nonautonomous equilibria in the scalar Example
\ref{Example_hull} provided our initial motivation for the present paper; cf.
Anagnostopoulou, P\"{o}tzsche, Rasmussen \cite{AgPR} for a treatise of
nonautonomous bifurcation theory. We wondered which controllability properties
would hold in the presence of the various bifurcation types of the
uncontrolled system; cf. Colonius and Kliemann \cite[Section 8.2]{ColK00} for
the autonomous case. It turned out that an adequate treatment would require an
appropriate framework of nonautonomous control systems and control flows,
which led to the present paper. We hope to come back to the bifurcation
problems for control systems. The rich properties of scalar nonautonomous
differential equations have found renewed interest, cf. Fabbri, Johnson, and
Mantellini \cite{FJM}, Campos, N\'{u}\~{n}ez, and Obaya \cite{CamNO23},
Due\~{n}as, N\'{u}\~{n}ez, and Obaya \cite{DNO23}, and Cheban \cite{Cheb24,
Cheb2}.

\section{Preliminaries\label{sec:preli}}

In this section we present basic properties of nonautonomous dynamical and
control systems. In particular, we explain in more detail the various
possibilities to describe systems of the form \eqref{control}.

A global real Borel measurable flow on a locally compact Hausdorff topological
space $X$ is a Borel measurable map $\phi:\mathbb{R}\times X\rightarrow X$
satisfying $\phi(0,x)=x$ and $\phi(t+s,x)=\phi(s,\phi(t,x))$ for all
$t,s\in\mathbb{R}$ and $x\in X$. The flow is continuous if $\phi$ is a
continuous map, and in this case we speak of a global real continuous flow or
continuos time dynamical system on $X$ denoted by $(X, \phi)$. We speak of
local flow if the map $\phi$ is defined, at least Boreal measurable, and
satisfies the two properties above on an open subset $\mathcal{O}%
\subset\mathbb{R}\times X$ containing $\{0\}\times X$ (see e.g. Ellis \cite{E}
and Johnson et al. \cite{JONNF16}).

Now we recall the definition of a (global and invertible) nonautonomous
dynamical system as a skew-product flow.

\begin{definition}
Let $B$ and $X$ be metric spaces. A skew product flow on the extended state
space $B\times X$ is a flow $\Phi:$ $\mathbb{R}\times B\times X\rightarrow
B\times X$ of the form%
\begin{equation}
\Phi(t,b,x):=(\theta(t,b),\varphi(t,b,x)), \label{eqS-P}%
\end{equation}
where $\theta:\mathbb{R}\times B\rightarrow B$ and $\varphi:\mathbb{R}\times
B\times X\rightarrow X$.
\end{definition}

The flow property of $\Phi$ is equivalent to the requirements that $\theta$ is
a flow on the base $B$ and the map $\varphi$ called cocycle satisfies%
\begin{align*}
&  i)\ \varphi(0,b,x)=x\ \mbox{for all}\ (b,x)\in B\times X,\\
&  ii)\ \varphi(t+s,b,x)=\varphi(s,\theta(t,b),\varphi
(t,b,x))\ \mbox{for all}\ t,s\in\mathbb{R},(b,x)\in B\times X.
\end{align*}
The skew product flows considered in this paper will be continuous, hence the
map $\Phi$ is continuous. Equivalently, the maps $\theta$ and $\varphi$ are
continuous. The autonomous dynamical system $\Phi$ on $B\times X$ defined by
\eqref{eqS-P} is called the skew product flow associated with the
nonautonomous dynamical system $(\theta,\varphi)$. The term skew product
emphasizes the asymmetric roles of the two components of the flow: the first
one, which is a flow on the base $B$ referred as the driving system, does not
depend on $x\in X$ (see e.g. Sacker and Sell \cite{SS78}, Kloeden and
Rasmussen \cite{KloR11}, and Cheban \cite{Cheb24}, \cite{Cheb2}).

Observe that the base of the skew product flow generated by the solutions of a
nonautonomous differential equation is related to the dependence on time of
the problem. It is a dynamical system that describes the changes in the
coefficient functions. In many applications, the base space is compact.

Considered systems of the form \eqref{control} are nonautonomous due to the
presence of the continuous flow $\sigma$ on the compact metric space $\Omega$
which is assumed to be minimal. Recall that a continuous flow $\phi$ on a
compact metric space $X$ is minimal, if it has no proper closed positively
invariant subsets. This is equivalent to the property that the flow has no
proper closed invariant subsets and to the property that the orbit of any
element of $\Omega$ is dense in $X$; cf. Akin, Auslander and Borg
\cite[Theorem 1.1]{AAB96}. The following result is due to Glasner and Weiss
\cite{GW93}, cf. \cite[Theorem 2.4]{AAB96}. It characterizes minimal flows on
compact metric spaces.

\begin{theorem}
\label{Theorem_minimal}Let $(X,\phi)$ be a continuous flow on a compact metric
space. If it is minimal, then it is either equicontinuous or sensitive with
respect to initial conditions, i.e., there is $\delta>0$ such that whenever
$U$ is a nonvoid open set there exist $x,y\in U$ such that $d(\phi
(T,x),\phi(T,y))>\delta$ for some $T>0$.
\end{theorem}

We note that the flow on the closure of an almost periodic function is equicontinuous.

Next we turn to nonautonomous control systems. Denote by $\varphi
(t,\omega,x_{0},u)$ the solution of the initial value problem $x(0)=x_{0}$ for
(\ref{control}) on the maximal open interval of existence $\mathcal{I}%
_{\omega,x_{0},u}$. The solution map in the extended state space $\Omega\times
M$ is denoted by%
\[
\psi(t,\omega,x_{0},u)=\bigl(\omega\cdot t,\varphi(t,\omega,x_{0},u)\bigr).
\]
We denote the distance on $\Omega$ as well as a distance on $M$ which is
compatible with the topology of $M$, by the letter $d$. Furthermore, the
metric on $\Omega\times M$ is
\[
d((\omega_{1},x_{1}),(\omega_{2},x_{2}))=\max\left\{  d(\omega_{1},\omega
_{2}),d(x_{1},x_{2})\right\}  .
\]
We call the nonautonomous differential equations with $u\equiv0$ the
uncontrolled system. This defines a continuous local flow%
\begin{equation}
\tau:\mathbb{R}\times\Omega\times M\rightarrow\Omega\times M,~\tau
(t,\omega,x_{0}):=\bigl(\omega\cdot t,\varphi(t,\omega,x_{0},0)\bigr).
\label{tau}%
\end{equation}
Denoting the time shift on $\mathcal{U}$ by $\theta_{t}u=u(t+\cdot
),t\in\mathbb{R}$, we obtain the local cocycle property%
\[
\varphi(t+s,\omega,x_{0},u)=\varphi\bigl(s,\omega\cdot t,\varphi
(t,\omega,x_{0},u),\theta_{t}u\bigr)\text{ where defined.}%
\]
The weak$^{\ast}$ topology on $\mathcal{U}$ is compact and metrizable; cf.
Kawan \cite[Proposition 1.14]{Kawan13}. Throughout this paper, we endow
$\mathcal{U}$ with a corresponding metric; cf. Lemma \ref{Lemma_metric}. The
map%
\begin{equation}
\Phi:\mathbb{R}\times\mathcal{U}\times\Omega\times M\rightarrow\mathcal{U}%
\times\Omega\times M,~\Phi(t,u,\omega,x_{0})=\bigl(\theta_{t}u,\psi
(t,\omega,x_{0},u)\bigr) \label{PHI}%
\end{equation}
satisfies $\Phi(0,u,\omega,x_{0})=(u,\omega,x_{0})$ and, where defined,%
\begin{align*}
\Phi(t+s,u,\omega,x_{0})  &  =(u(t+s+\cdot),\omega\cdot(t+s),\varphi
(t+s,\omega,x_{0},u))\\
&  =(u(t+s+\cdot),(\omega\cdot s)\cdot t),\varphi(t,\omega\cdot s,\varphi
(s,\omega,x_{0},u),u(s+\cdot))\\
&  =\Phi(t,\Phi(s,u,\omega,x_{0})).
\end{align*}
We also write $\Phi_{t}(u,\omega,x_{0})=\Phi(t,u,\omega,x_{0})$. The map
$\Phi$ defines a local skew product flow, called \emph{local control flow}.
The following theorem, which is a variant of Kawan \cite[Proposition
1.17]{Kawan13}, describes the continuity properties of $\Phi$.

\begin{theorem}
Consider a control system of the form (\ref{control}). Let $\mathcal{U}$ be
endowed with a metric compatible with the weak$^{\ast}$ topology on
$L^{\infty}(\mathbb{R},\mathbb{R}^{m})$.

(i) Then the shift flow $\theta:\mathbb{R}\times\mathcal{U}\rightarrow
\mathcal{U}$ is continuous.

(ii) The local cocycle $\varphi:\mathbb{R}\times\Omega\times M\times
\mathcal{U}\rightarrow M$ is continuous in the following sense: For
$(\omega^{\ast},x^{\ast},u^{\ast})\in\Omega\times M\times\mathcal{U}$, let
$t^{\ast}$ be in the maximal open interval of existence $\mathcal{I}%
_{\omega^{\ast},x^{\ast},u^{\ast}}$. Suppose that for $(t,\omega,x,u)$ in a
neighborhood of $(t^{\ast},\omega^{\ast},x^{\ast},u^{\ast})$ one has that
$t\in\mathcal{I}_{\omega,x,u}$. Then for any sequence $(t^{n},\omega^{n}%
,x^{n},u^{n})\rightarrow(t^{\ast},\omega^{\ast},x^{\ast},u^{\ast})$ in
$\mathbb{R}\times\Omega\times M\times\mathcal{U}$ it follows that%
\[
\varphi(t^{n},\omega^{n},x^{n},u^{n})\rightarrow\varphi(t^{\ast},\omega^{\ast
},x^{\ast},u^{\ast})\text{ for }n\rightarrow\infty.
\]

(iii) The local control flow $\Phi$ defined in (\ref{PHI}) is continuous: In
the situation of (ii) it follows that%
\[
\Phi(t^{n},u^{n},\omega^{n},x^{n})\rightarrow\Phi(t^{\ast},u^{\ast}%
,\omega^{\ast},x^{\ast})\text{ in }\mathcal{U}\times\Omega\times M.
\]

\end{theorem}

\begin{proof}
Assertion (iii) is an immediate consequence of (i) and (ii). Assertion (i)
holds by \cite[Proposition~1.15]{Kawan13}. We sketch the proof of (ii)
following the proof of \cite[Proposition~1.17]{Kawan13}. Standard arguments
allow us to suppose that $M=\mathbb{R}^{d}$.

Step 1. Fix $\tau>0$ in $\mathcal{I}_{\omega^{\ast},x^{\ast},u^{\ast}}$ and
consider, for sequences $\omega^{n}\rightarrow\omega^{\ast}$ in $\Omega
,u^{n}\rightarrow u^{\ast}$ in $\mathcal{U}$, and $x^{n}\rightarrow x^{\ast}$
in $\mathbb{R}^{d}$, the corresponding solutions $\xi^{n}(t):=\varphi
(t,\omega^{n},x^{n},u^{n})$ on $[0,\tau]$. In the first two steps of the
proof, let us assume that there exists a compact set $K\subset\mathbb{R}^{d}$
with $\xi^{n}(t)\in K$ for all $n\in\mathbb{N}$ and $t\in\lbrack0,\tau]$. We
show that the set $\left\{  \xi^{n}\right\}  _{n\in\mathbb{N}}$ is relatively
compact in $C([0,\tau];\mathbb{R}^{d})$ endowed with the sup-norm. Let $0\leq
t_{1}<t_{2}\leq\tau$. Then%
\[
\left\Vert \xi^{n}(t_{2})-\xi^{n}(t_{1})\right\Vert \leq\int_{t_{1}}^{t_{2}%
}\left(  \left\Vert f_{0}(\omega^{n}\cdot s,\xi^{n}(s))\right\Vert +\sum
_{i=1}^{m}\left\Vert u_{i}^{n}(s)\right\Vert \left\Vert f_{i}(\omega^{n}\cdot
s,\xi^{n}(s))\right\Vert \right)  ds.
\]
Since the set $\Omega\times K\times U$ is compact it follows that the set
$\left\{  \xi^{n}\right\}  _{n\in\mathbb{N}}$ is equicontinuous. From the
assumption that $\xi^{n}(t)\in K$ it follows that for each $t\in\lbrack
0,\tau]$ the set $\left\{  \xi^{n}(t)\right\}  _{n\in\mathbb{N}}$ is
relatively compact. Hence, the Arzel\'{a}-Ascoli theorem can be applied and
there exists a convergent subsequence $\xi^{k_{n}}\rightarrow\xi^{0}\in
C([0,\tau];\mathbb{R}^{d})$. The same arguments hold for $\tau<0$ in
$\mathcal{I}_{\omega^{\ast},x^{\ast},u^{\ast}}$.

Step 2. We claim that $\xi^{0}(t)=\varphi(t,\omega^{\ast},x^{\ast},u^{\ast})$
for all $t$. This follows similarly as Step 2 in \cite[Proposition~1.17]%
{Kawan13} using weak$^{\ast}$ convergence of $u^{n}\rightarrow u^{\ast}$.

Step 3. For $t^{n}\rightarrow t^{\ast},$ $\omega^{n}\rightarrow\omega^{\ast}$
in $\Omega$, $x^{n}\rightarrow x^{\ast}$ in $\mathbb{R}^{d}$, and
$u^{n}\rightarrow u^{\ast}$ in $\mathcal{U}$, it follows that $\varphi
(t^{n},\omega^{n},x^{n},u^{n})\rightarrow\varphi(t^{\ast},\omega^{\ast
},x^{\ast},u^{\ast})$. This uses a continuously differentiable cut-off
function $\chi:\mathbb{R}^{d}\rightarrow\lbrack0,1]$ with $\chi(x)\equiv1$ on
$K$ and $\chi(x)\equiv0$ on the complement of another compact set
$\widetilde{K}\supset K$. For details see Step 3 in \cite[Proposition~1.17]%
{Kawan13}.
\end{proof}

The flow $\Phi$ can be considered in three different ways as a skew product flow:

(i) Let the base space be $\mathcal{U}\times\Omega$ with base flow $\Theta
_{t}(u,\omega)=(\theta_{t}u,\sigma_{t}(\omega)),t\in\mathbb{R}$, and cocycle
on $M$ given by%
\begin{equation}
\varphi_{1}(t,x,(u,\omega))=\varphi(t,\omega,x,u). \label{skew1}%
\end{equation}

(ii) Let the base space be $\mathcal{U}$ with base flow $\theta_{t}%
u,t\in\mathbb{R}$, and cocycle on $\Omega\times M$ given by%
\begin{equation}
\varphi_{2}(t,\left(  \omega,x\right)  ,u))=\psi(t,\omega,x,u)=(\omega\cdot
t,\varphi(t,\omega,x,u)). \label{skew2}%
\end{equation}

(iii) Let the base space be $\Omega$ with base flow $\sigma_{t}(\omega
),t\in\mathbb{R}$, and cocycle on $\mathcal{U}\times M$ given by%
\begin{equation}
\Phi_{1}(t,\left(  u,x\right)  ,\omega))=(\theta_{t}u,\varphi(t,\omega,x,u)).
\label{skew3}%
\end{equation}
Here the presence of $\omega$ indicates that the cocycle $\Phi_{1}$ may be
viewed as a nonautonomous control flow on $\mathcal{U}\times M$. Note that, in
all three cases, the base flows are globally defined.

For the following scalar example, Elia, Fabbri, and N\'{u}\~{n}ez \cite{EFN25}
analyzed the bifurcation behavior of the uncontrolled system (for
$u(t)\equiv0$) with respect to $\varepsilon>0$. Here, with the so called hull
construction due to Bebutov \cite{Be}, they passed from a single equation to a
family of equations; cf. Sell \cite{Sell1}.

\begin{example}
\label{Example_hull}Consider the system in $\mathbb{R}$ given by
\begin{equation}
\dot{x}(t)=-x^{3}(t)+\overline{c}(t)x^{2}(t)+\varepsilon(\overline
{b}(t)x(t)+\overline{a}(t))+u(t),u(t)\in U=[\rho_{1},\rho_{2}],
\label{eq:esempio}%
\end{equation}
where $(\overline{a},\overline{b},\overline{c})$ are bounded uniformly
continuous real functions and $\varepsilon>0$ and $\rho_{1}<0<\rho_{2}$ are
constants. For the uncontrolled system the skew product formalism defines a
(possibly local) real continuous flow $\tau$ on the vector bundle
$\Omega\times\mathbb{R}$, where $\Omega$ is the hull of $(\overline
{a},\overline{b},\overline{c})$. That is $\Omega$ is the closure in the
compact-open topology of $C(\mathbb{R},\mathbb{R}^{3})$ of the set of
time-shifts%
\[
\left\{  (\overline{a}(t+\cdot),\overline{b}(t+\cdot),\overline{c}%
(t+\cdot))\left\vert t\in\mathbb{R}\right.  \right\}  .
\]
Define $a(\omega):=\omega_{1}(0),b(\omega)=\omega_{2}(0)$, and $c(\omega
)=\omega_{3}(0)$ for $\omega=(\omega_{1},\omega_{2},\omega_{3})\in\Omega$ and
write the time shift as $\omega(t+\cdot)=\omega\cdot t,t\in\mathbb{R}$. We
obtain the family of equations%
\begin{equation}
\dot{x}(t)=-x^{3}(t)+c(\omega\cdot t)x^{2}(t)+\varepsilon\left(  b(\omega\cdot
t)x(t)+a(\omega\cdot t)\right)  +u(t),\quad\omega\in\Omega.
\label{family-omega}%
\end{equation}
Thus the original equation (\ref{eq:esempio}), which is the equation with
$\overline{\omega}=(\overline{a},\overline{b},\overline{c})$, is embedded into
a family of equations. Additional recurrence assumptions on the coefficient
functions guarantee that the flow $(\Omega,\sigma)$ with $\sigma
(t,\omega):=\omega(t+\cdot),t\in\mathbb{R}$, is a minimal flow on a compact
metric space. Let $\varphi(\cdot,\omega,x,u)$ be the local solution of
(\ref{family-omega}). With the notation introduced in this section, the map
\[
\Phi:\mathbb{R}\times\mathcal{U}\times\Omega\times\mathbb{R}\rightarrow
\mathcal{U}\times\Omega\times\mathbb{R}%
\]
given by $\Phi(t,u,\omega,x):=\bigl(\theta_{t}u,\omega\cdot t,\varphi
(t,\omega,x,u)\bigr)$ is a continuous dynamical system on the extended state
space $\mathcal{U}\times\Omega\times\mathbb{R}$, the local control flow
corresponding to (\ref{family-omega}). Remark \ref{Remark_hull} gives some
information on the controllability properties of systems of the form
(\ref{family-omega}).
\end{example}

\section{Chain Control Sets\label{chain}}

In this section we define and characterize chain control sets in the extended
state space $\Omega\times M$. For simplicity of exposition, we suppose here
that $\Phi$ is a global flow.

It will be convenient to write for a subset $A\subset\Omega\times M$ the
section with a fiber over $\omega\in\Omega$ as%
\[
A_{\omega}:=A\cap(\{\omega\}\times M)\text{, hence }A=\bigcup\nolimits_{\omega
\in\Omega}A_{\omega}.
\]
Where convenient, we identify $A_{\omega}$ and $\{x\in M\left\vert
(\omega,x)\in A\right.  \}$.

\begin{definition}
Fix $(\omega,x),(\overline{\omega},y)\in\Omega\times M$ and let $\varepsilon
,T>0$. A controlled $(\varepsilon,T)$-chain $\zeta$ from $(\omega,x)$ to
$(\overline{\omega},y)$ is given by $n\in\mathbb{N}$, elements $(\omega
_{0},x_{0})=(\omega,x),\allowbreak(\omega_{1},x_{1}),\dots,(\omega_{n}%
,x_{n})=(\omega_{n},y)\in\Omega\times M$, controls $u_{0},\dots,u_{n-1}%
\in\mathcal{U}$, and times $T_{0},\dots,T_{n-1}\geq T$ such that

(i) $\omega_{j}\cdot T_{j}=\omega_{j+1}$ for $j=0,\dots,n-1$, and
$d(\omega_{n},\overline{\omega})<\varepsilon$,

(ii) $d(\varphi(T_{j},\omega_{j},x_{j},u_{j}),x_{j+1})<\varepsilon$ for
$j=0,\dots,n-1$.
\end{definition}

With $S_{0}:=0,S_{j}:=T_{0}+\cdots+T_{j-1},j=1,\dots,n$, we can write the
conditions above as%
\[
\omega_{j}=\omega\cdot S_{j},d(\omega\cdot S_{n},\overline{\omega
})<\varepsilon,\text{ }d(\varphi(T_{j},\omega\cdot S_{j},x_{j},u_{j}%
),x_{j+1})<\varepsilon\text{ for }j=0,\dots,n-1.
\]
A nonvoid set $A\subset\Omega\times M$ is called chain controllable, if for
all $(\omega,x),(\overline{\omega},y)\in A$ and all $\varepsilon,T>0$ there
exists a controlled $(\varepsilon,T)$-chain in $\Omega\times M$ from
$(\omega,x)$ to $(\overline{\omega},y)$. If for all $\varepsilon,T>0$ all
segments $\varphi(t,\omega_{j},x_{j},u_{j}),t\in\lbrack0,T_{j}]$, of the
controlled $(\varepsilon,T)$-chains are contained in a subset $Q\subset M$, we
say that $A$ is chain controllable in $\Omega\times Q$.

This definition serves to introduce the following concept of (nonautonomous)
chain control sets.

\begin{definition}
\label{Etotal}A chain control set is a nonvoid maximal set $E\subset
\Omega\times M$ such that

(i) for all $(\omega,x)\in E$ there is $u\in\mathcal{U}$ with $\psi
(t,\omega,x,u)\in E$ for all $t\in\mathbb{R}$,

(ii) for all $(\omega,x),(\overline{\omega},y)\in E$ and all $\varepsilon,T>0$
there exists a controlled $(\varepsilon,T)$-chain from $(\omega,x)$ to
$(\overline{\omega},y)$.
\end{definition}

Note that, for chain control sets, the three components $x$, $\omega$, and $u$
are treated in different ways: jumps are allowed in $x$, approximate
reachability is required for $\omega$ and no condition on the controls is imposed.

\begin{remark}
\label{Remark_2T}We may assume that, for any controlled $(\varepsilon
,T)$-chain, the jump times satisfy $T_{j}\in\lbrack T,2T]$ for all $j$. This
can be achieved by introducing trivial jumps at times which are of the form
$kT\leq T_{j}$ with $k\in\mathbb{N}$ till the remaining time is $T_{j}-kT<2T$.
\end{remark}

In general, the concatenation of controlled $(\varepsilon,T)$-chains is not a
controlled $(\varepsilon,T)$-chain. This is due to the requirement
$d(\omega_{n},\overline{\omega})<\varepsilon$. Instead, the following weaker
property holds.

\begin{lemma}
\label{Lemma_concatenation}Consider $(\omega^{i},x^{i})\in\Omega\times
M,i=1,2,3$, and assume that for all $\varepsilon,T>0$ there are controlled
$(\varepsilon,T)$-chains from $(\omega^{1},x^{1})$ to $(\omega^{2},x^{2})$ and
from $(\omega^{2},x^{2})$ to $(\omega^{3},x^{3})$. Then, for all
$\varepsilon,T>0$, there are controlled $(\varepsilon,T)$-chains from
$(\omega^{1},x^{1})$ to $(\omega^{3},x^{3})$.
\end{lemma}

\begin{proof}
Let $\varepsilon,T>0$. There is a controlled $(\varepsilon/2,T)$-chain from
$(\omega^{2},x^{2})$ to $(\omega^{3},x^{3})$ with times $T_{0}^{2}%
,\dots,T_{n_{2}-1}^{2}\geq T$ and controls $u_{0}^{2},\dots,u_{n_{2}-1}^{2}%
\in\mathcal{U}$. With $S_{0}^{2}:=0,S_{j}^{2}:=T_{0}^{2}+\cdots+T_{j-1}%
^{2},j=1,\dots,n_{2}$, we get
\[
(\omega_{0},x_{0}^{2})=(\omega^{2},x^{2}),(\omega^{2}\cdot S_{1}^{2},x_{1}%
^{2}),\dots,(\omega^{2}\cdot S_{n_{2}}^{2},x_{n_{2}}^{2})=(\omega^{2}\cdot
S_{n_{2}}^{2},x^{3})\in\Omega\times M,
\]
such that $d(\omega^{2}\cdot S_{n_{2}}^{2},\omega^{3})<\varepsilon/2$, and
\[
d(\varphi(T_{j}^{2},\omega^{2}\cdot S_{j}^{2},x_{j}^{2},u_{j}^{2}),x_{j+1}%
^{2})<\varepsilon/2\text{ for }j=0,\dots,n_{2}-1.
\]
By continuity there is $\delta\in(0,\varepsilon)$ such that $d(\omega^{\prime
},\omega^{2})<\delta$ implies $d(\omega^{\prime}\cdot S_{n_{2}}^{2},\omega
^{2}\cdot S_{n_{2}}^{2})<\varepsilon/2$ and, for $j=0,\dots,n-1$,%
\[
d\left(  \varphi(T_{j}^{2},\omega^{\prime}\cdot S_{j}^{2},x_{j}^{2},u_{j}%
^{2}),\varphi(T_{j}^{2},\omega^{2}\cdot S_{j}^{2},x_{j}^{2},u_{j}^{2})\right)
<\varepsilon/2.
\]
It follows that a controlled $(\varepsilon,T)$-chain $\zeta^{2}$ from
$(\omega^{\prime},x^{2})$ to $(\omega^{3},x^{3})$ is given by $T_{j}^{2}$ and
$u_{j}^{2}$ as above and
\[
(\omega_{0}^{\prime},x_{0}^{2})=(\omega^{\prime},x^{2}),(\omega^{\prime}\cdot
S_{1}^{2},x_{1}^{2}),\dots,(\omega^{\prime}\cdot S_{n_{2}}^{2},x_{n}%
^{2})=(\omega^{\prime}\cdot S_{n_{2}}^{2},x^{3}),
\]
such that
\[
d(\omega^{\prime}\cdot S_{n_{2}}^{2},\omega^{3})\leq d(\omega^{\prime}\cdot
S_{n_{2}}^{2},\omega^{2}\cdot S_{n_{2}}^{2})+d(\omega^{2}\cdot S_{n_{2}}%
^{2},\omega^{3})<\varepsilon/2+\varepsilon/2=\varepsilon,
\]
and for $j=0,\dots,n_{2}-1$%
\begin{align*}
&  d(\varphi(T_{j}^{2},\omega^{\prime}\cdot S_{j}^{2},x_{j}^{2},u_{j}%
^{2}),x_{j+1}^{2})\\
&  \leq d(\varphi(T_{j}^{2},\omega^{\prime}\cdot S_{j}^{2},x_{j}^{2},u_{j}%
^{2}),\varphi(T_{j}^{2},\omega^{2}\cdot S_{j}^{2},x_{j}^{2},u_{j}%
^{2}))+d(\varphi(T_{j}^{2},\omega^{2}\cdot S_{j}^{2},x_{j}^{2},u_{j}%
^{2}),x_{j+1}^{2})\\
&  <\varepsilon/2+\varepsilon/2=\varepsilon.
\end{align*}
There is a controlled $(\delta,T)$-chain $\zeta^{1}$ from $(\omega^{1},x^{1})$
to $(\omega^{2},x^{2})$ given by%
\[
T_{0}^{1},\dots,T_{n_{1}-1}^{1}\geq T,u_{0}^{1},\dots,u_{n_{1}-1}^{1}%
\in\mathcal{U}\text{,}%
\]
and, with $\,S_{0}^{1}:=0,S_{j}^{1}:=T_{0}^{1}+\cdots+T_{j-1}^{1}%
,j=1,\dots,n_{1}$, one has $d(\omega^{1}\cdot S_{n_{1}}^{1},\omega^{2}%
)<\delta$ and%
\begin{align*}
\left.  (\omega_{0}^{1},x_{0}^{1})=\right.  (\omega^{1},x^{1}),(\omega
^{1}\cdot S_{1}^{1},x_{1}^{1}),\dots,(\omega^{1}\cdot S_{n_{1}}^{1},x_{n_{1}%
}^{1})  &  =(\omega^{1}\cdot S_{n_{1}}^{1},x^{2}),\\
\text{and }d(\varphi(T_{j}^{1},x_{j}^{1},\omega^{1}\cdot S_{j}^{1},u_{j}%
^{1}),x_{j+1}^{1})  &  <\delta\text{ for }j=0,\dots,n_{1}-1.
\end{align*}
Since $\delta<\varepsilon$ the chain $\zeta^{1}$ is also a controlled
$(\varepsilon,T)$-chain from $(\omega^{1},x^{1})$ to $(\omega^{2},x^{2})$.
Furthermore, since $d(\omega^{1}\cdot S_{n_{1}}^{1},\omega^{2})<\delta$ we may
choose $(\omega^{\prime},x^{2})$ with $\omega^{\prime}:=\omega^{1}\cdot
S_{n_{1}}^{1}$ as starting point for the chain $\zeta^{2}$. Then the
concatenation $\zeta^{2}\circ\zeta^{1}$ is a controlled $(\varepsilon
,T)$-chain from $(\omega^{1},x^{1})$ to $(\omega^{3},x^{3})$.
\end{proof}

A consequence of Lemma \ref{Lemma_concatenation} is the following result.
Recall from (\ref{tau}) that $\tau$ denotes the local flow on $\Omega\times M$
of the uncontrolled system. If $\emptyset\not =\mathcal{K}\subset\Omega\times
M$ is compact and $\tau$-invariant it is called minimal $\tau$-invariant set,
if the restriction of $\tau$ to $\mathcal{K}$ is minimal.

\begin{proposition}
\label{Proposition_contained}(i) Every chain controllable set $E^{0}%
\subset\Omega\times M$ is contained in a maximal chain controllable set.

(ii) Every minimal $\tau$-invariant set $\mathcal{K}\subset\Omega\times M$ is
contained in a maximal chain controllable set.
\end{proposition}

\begin{proof}
(i) Define $E^{\prime}$ as the union of all chain controllable sets containing
$E^{0}$. Let $(\omega^{1},x^{1}),(\omega^{3},x^{3})\in E^{\prime}$ and
$(\omega^{2},x^{2})\in E^{0}$. Then, for all $\varepsilon,T>0$, there are
controlled $(\varepsilon,T)$-chains $\zeta^{1}$ and $\zeta^{2}$ from
$(\omega^{1},x^{1})$ to $(\omega^{2},x^{2})$ and from $(\omega^{2},x^{2})$ to
$(\omega^{3},x^{3})$, respectively. By Lemma \ref{Lemma_concatenation} one
finds for all $\varepsilon,T>0$ controlled $(\varepsilon,T)$-chains from
$(\omega^{1},x^{1})$ to $(\omega^{3},x^{3})$. Hence $E^{\prime}$ is chain
controllable and certainly it is maximal with this property.

(ii) Since every minimal $\tau$-invariant set is chain controllable (with
control $u(t)\equiv0$), the assertion follows by (i).
\end{proof}

In Proposition \ref{Proposition_contained2} we will sharpen the assertions in
Proposition \ref{Proposition_contained} by showing that any chain controllable
set is contained in a chain control set.

The following arguments are similar to those in Kawan's proof (cf.\cite[Proposition 1.24(i)]{Kawan13}) that chain control sets are closed.

\begin{proposition}
\label{Proposition_closed}Let $E$ be a chain control set in $\Omega\times M$.
Then the fibers $E_{\omega},\omega\in\Omega$, of $E$ are closed.
\end{proposition}

\begin{proof}
We prove the assertion by showing that the set $E^{1}:=\bigcup_{\omega
\in\Omega}\mathrm{cl}E_{\omega}$ satisfies the properties (i) and (ii) of
chain control sets. By maximality of $E$, it then follows that $E=E^{1}$, and
hence $E_{\omega}=\mathrm{cl}E_{\omega}$ for all $\omega\in\Omega$.

(i) For every $x\in\mathrm{cl}E_{\omega}$ there is a sequence $x_{n}\in
E_{\omega}$ converging to $x$. By property (i) of chain control sets, for
every $x_{n}$ there exists $u_{n}\in\mathcal{U}$ with $\psi(t,\omega
,x_{n},u_{n})\in E$ for all $t\in\mathbb{R}$. By compactness of $\mathcal{U}$
we may assume that $u_{n}$ converges to $u_{0}\in\mathcal{U}$. Then continuity
implies $\varphi(t,\omega,x,u_{0})\in\mathrm{cl}E_{\omega\cdot t}$ for all
$t\in\mathbb{R}$. Hence, $E^{1}$ satisfies property (i).

(ii) Let $(\omega,x),(\overline{\omega},y)\in E^{1}$ and $\varepsilon,T>0$. As
shown in (i) there is a control function $u_{0}\in\mathcal{U}$ such that
$(\omega_{1},x_{1}):=\psi(T,\omega,x,u_{0})\in E^{1}$ with $\omega_{1}%
=\omega\cdot T$. Since $x_{1}\in\mathrm{cl}E_{\omega\cdot T}$ there is
$x_{2}\in E_{\omega\cdot T}$ with $d(x_{1},x_{2})<\varepsilon$ and
$y\in\mathrm{cl}E_{\overline{\omega}}$ implies that there is $y_{2}\in
E_{\overline{\omega}}$ with $d(y,y_{2})<\varepsilon/2$. Then there is a
controlled $(\varepsilon/2,T)$-chain $\zeta$ from $(\omega\cdot T,x_{2})$ to
$(\overline{\omega},y_{2})$. Now add the point $(\omega,x)$ with control
$u_{0}$ on $[0,T]$ in the beginning of the controlled chain $\zeta$ and
replace the final point by $y$. This is a controlled $(\varepsilon,T)$-chain
from $(\omega,x)$ to $(\overline{\omega},y)$. Hence, $E^{1}$ also satisfies
property (ii).
\end{proof}

If the flow $\sigma$ on $\Omega$ is equicontinuous (cf. Theorem
\ref{Theorem_minimal}), the following stronger result holds.

\begin{proposition}
\label{Proposition_closed2}Assume that the minimal flow $\sigma$ on $\Omega$
is equicontinuous. Let $E\subset\Omega\times Q$ be a chain control set, where
$Q\subset M$ is compact, and suppose that $E$ is chain controllable in
$\Omega\times Q$. Then it follows that $E$ is compact.
\end{proposition}

\begin{proof}
We prove the assertion by showing that the set $\mathrm{cl}E$ satisfies the
properties (i) and (ii) of chain control sets. By maximality of $E$, it then
follows that $E=\mathrm{cl}E$ and hence $E$ is compact.

(i) For every $(\omega,x)\in\mathrm{cl}E$ there is a sequence $(\omega
_{n},x_{n})\in E$ converging to $(\omega,x)$. For every $(\omega_{n},x_{n})$
there exists $u_{n}\in\mathcal{U}$ with $\psi(t,\omega_{n},x_{n},u_{n})\in E$
for all $t\in\mathbb{R}$. By compactness of $\mathcal{U}$ we may assume that
$u_{n}$ converges to $u_{0}\in\mathcal{U}$. Then continuity of $\psi$ implies
$\psi(t,\omega,x,u_{0})\in\mathrm{cl}E$ for all $t\in\mathbb{R}$. Hence,
$\mathrm{cl}E$ satisfies (i).

(ii) Let $(\omega,x),(\overline{\omega},y)\in\mathrm{cl}E$ and $\varepsilon
,T>0$. By continuity of $\psi$ and compactness of $\Omega,Q$, and
$\mathcal{U}$ there is $\delta\in(0,\varepsilon/3)$ such that for all $z\in
Q,\omega^{\prime},\omega^{\prime\prime}\in\Omega$, and $u\in\mathcal{U}$%
\begin{equation}
d(\omega^{\prime},\omega^{\prime\prime})<\delta\text{ implies }d\left(
\varphi(t,\omega^{\prime},z,u),\varphi(t,\omega^{\prime\prime},z,u)\right)
<\varepsilon/3\text{ for }t\in\lbrack0,2T]. \label{uniform0}%
\end{equation}
By (i) there is a control function $u_{0}\in\mathcal{U}$ such that
$(\omega\cdot T,x_{1}):=\psi(T,\omega,x,u_{0})\allowbreak\in\mathrm{cl}E$.
Since $(\omega\cdot T,x_{1})\in\mathrm{cl}E$ there is $(\omega_{2},x_{2})\in
E$ with $d((\omega\cdot T,x_{1}),(\omega_{2},x_{2}))<\delta$. Similarly, for
$(\overline{\omega},y)\in\mathrm{cl}E$ there is $(\widetilde{\omega
},\widetilde{y})\in E$ with $d((\overline{\omega},y),(\widetilde{\omega
},\widetilde{y}))<\varepsilon/3$. There is a controlled $(\delta,T)$-chain
$\zeta$ from $(\omega_{2},x_{2})$ to $(\widetilde{\omega},\widetilde{y})$
given by $(\omega_{2},x_{2}),\allowbreak(\omega_{3},x_{3}),\dots,(\omega
_{n},x_{n})=(\omega_{n},\widetilde{y})\in\Omega\times M$, controls
$u_{2},\dots,u_{n-1}\in\mathcal{U}$, and times $T_{2},\dots,T_{n-1}\geq T$. By
assumption, the points $x_{i}$ may be chosen in $Q$. By Remark \ref{Remark_2T}%
, we may suppose that the jump times $T_{j}$ are in $[T,2T]$.

Now add the point $(\omega,x)$ with control $u_{0}$ on $[0,T]$ in the
beginning of the controlled chain $\zeta$ and replace the final point by
$(\overline{\omega},y)$. Since $d(\omega\cdot T,\omega_{2})<\delta$
equicontinuity of $\sigma$ implies that $d(\omega\cdot(T+S_{j}),\omega
_{2}\cdot S_{j})<\delta$ for all $j$, and%
\begin{align*}
d(\omega\cdot(T+S_{n}),\overline{\omega})  &  \leq d(\omega\cdot
(T+S_{n}),\omega_{2}\cdot S_{n})+d(\omega_{2}\cdot S_{n},\widetilde{\omega
})+d(\widetilde{\omega},,\overline{\omega})\\
&  <\delta+\delta+\varepsilon/3<\varepsilon.
\end{align*}
Since $x_{j}\in Q$ for all $j$ it follows by (\ref{uniform0}) that%
\begin{align*}
&  d\left(  \varphi(T_{j},\omega\cdot\left(  T+S_{j}\right)  ,x_{j}%
,u_{j}),x_{j+1}\right) \\
&  \leq d\left(  \varphi(T_{j},\omega\cdot\left(  T+S_{j}\right)  ,x_{j}%
,u_{j}),\varphi(T_{j},\omega_{2}\cdot S_{j},x_{j},u_{j})\right) \\
&  \qquad+d\left(  \varphi(T_{j},\omega_{2}\cdot S_{j},x_{j},u_{j}%
),x_{j+1}\right)  <\varepsilon/3+\delta<\varepsilon.
\end{align*}
Thus, this yields a controlled $(\varepsilon,T)$-chain from $(\omega,x)$ to
$(\overline{\omega},y)$, and hence $\mathrm{cl}E$ also satisfies property (ii)
of chain control sets.
\end{proof}

If $\sigma$ is not equicontinuous, Theorem \ref{Theorem_minimal} implies that
it is sensitive with respect to initial conditions. Hence one will not expect
that chain control sets are closed. An example of a minimal equicontinuous
flow is the Kronecker flow on the torus.

\begin{example}
Let $\Omega=\mathbb{T}^{2}$ be the $2$-torus and let $\gamma=(\gamma
_{1},\gamma_{2})\in\mathbb{R}^{2}$ be a vector whose components are rationally
independent (e.g. $\gamma=(1,\sqrt{2})$). For $\omega\in\Omega$, write
$\omega=(\exp2\pi\imath\psi_{1},\exp2\pi\imath\psi_{2})$ and define, for each
$t\in\mathbb{R}$, the map
\[
\tau(t,\omega)=\tau_{t}(\omega)=(\exp2\pi\imath(\psi_{1}+\gamma_{1}t),\exp
2\pi\imath(\psi_{2}+\gamma_{2}t)).
\]
We write $\psi=(\psi_{1},\psi_{2})$ and $\tau(t,\psi)=\tau_{t}(\psi
)=\psi+\omega t=(\psi_{1}+\omega_{1}t,\psi_{2}+\omega_{2}t)$. One says that
$(\Omega,\tau)$ is a quasi-periodic flow or a Kronecker flow (or a Kronecker
winding) on the torus. It is minimal and almost periodic. The real numbers
$\gamma_{1},\gamma_{2}$ are called the frequencies of the flow. This is an
example of a continuous minimal and uniquely ergodic flow.

\end{example}


We turn to analyze chain reachability.

\begin{definition}
The chain reachable set from $(\omega,x)\in\Omega\times M$ is%
\[
\mathbf{R}^{c}(\omega,x)=\left\{  (\overline{\omega},y)\in\Omega\times
M\left\vert
\begin{array}
[c]{c}%
\forall\varepsilon,T>0~\exists~\text{controlled }(\varepsilon,T)\text{-chain}%
\\
\text{ from }(\omega,x)\text{ to }(\overline{\omega},y)
\end{array}
\right.  \right\}  .
\]

\end{definition}

\begin{proposition}
\label{Proposition_inv}(i) Let $(\omega,x)\in\Omega\times M$. Then for all
$(\overline{\omega},y)\in\mathbf{R}^{c}(\omega,x)$ and all $u\in\mathcal{U}$
it follows that $(\overline{\omega}\cdot\tau,\varphi(\tau,\overline{\omega
},y,u))\in\mathbf{R}^{c}(\omega,x)$ for all $\tau>0$. Furthermore, there
exists $v\in\mathcal{U}$ such that $(\overline{\omega}\cdot\tau,\varphi
(\tau,\overline{\omega},y,v))\in\mathbf{R}^{c}(\omega,x)$ for all $\tau<0$.

(ii) Let $F\subset\Omega\times Q$ be a maximal chain controllable set in
$\Omega\times Q$ where $Q\subset M$ is compact. Then, for all $(\omega,x)\in
F$ and $\tau\in\mathbb{R}$, there exists $v\in\mathcal{U}$ such that
$(\omega\cdot\tau,\varphi(\tau,\omega,x,v))\in F$. It follows that $F$ is a
chain control set.
\end{proposition}

\begin{proof}
(i) Let $(\overline{\omega},y)\in\mathbf{R}^{c}(\omega,x)$ and $u\in
\mathcal{U}$. Fix $\tau>0$ and $\varepsilon,T>0$. By continuity of $\varphi$,
there is $\delta\in(0,\varepsilon)$ such that $d(\omega^{\prime}%
,\overline{\omega})<\delta$ and $d(y^{\prime},y)<\delta$ implies that%
\begin{equation}
\label{DELTA}
d(\omega^{\prime}\cdot\tau,\overline{\omega}\cdot\tau)<\varepsilon\text{ and
}d\left(  \varphi(\tau,\omega^{\prime},y^{\prime},u),\varphi(\tau
,\overline{\omega},y,u)\right)  <\varepsilon.
\end{equation}
There is a controlled $(\delta,T)$-chain $\zeta$ from $(\omega,x)$ to
$(\overline{\omega},y)$. We prolong the final segment of $\zeta$ by defining
$T_{n-1}^{\prime}:=T_{n-1}+\tau,\omega_{n}^{\prime}:=\omega_{n}\cdot\tau$, and
by defining a control $u_{n-1}^{\prime}$ by%
\[
u_{n-1}^{\prime}(t):=\left\{
\begin{array}
[c]{lll}%
u_{n-1}(t) & \text{for} & t\in\lbrack0,T_{n-1}]\\
u(t-T_{n-1}) & \text{for} & t\in(T_{n-1},T_{n-1}+\tau]
\end{array}
\right.  .
\]
Since $d(\omega_{n},\overline{\omega})<\delta$ and $d\left(  \varphi
(T_{n-1},\omega_{n-1},x_{n-1},u_{n-1}),y\right)  <\delta$ it follows by \eqref{DELTA} that
$d(\omega_{n}^{\prime},\overline{\omega}\cdot\tau)=d(\omega_{n}\cdot
\tau,\overline{\omega}\cdot\tau)<\varepsilon$ and%
\begin{align*}
&  d(\varphi(T_{n-1}^{\prime},\omega_{n-1},x_{n-1},u_{n-1}^{\prime}%
),\varphi(\tau,\overline{\omega},y,u))\\
&  =d(\varphi(\tau,\omega_{n},\varphi(T_{n-1},\omega_{n-1},x_{n-1}%
,u_{n-1}),u),\varphi(\tau,\overline{\omega},y,u))<\varepsilon.
\end{align*}
Hence, with this new final segment, we obtain a controlled $(\varepsilon
,T)$-chain from $(\omega,x)$ to $\allowbreak(\overline{\omega}\cdot
\tau,\varphi(\tau,\overline{\omega},y,u))$.

In the case of negative time $\tau$, consider for a sequences $\varepsilon
^{k}\rightarrow0$ and $T^{k}\rightarrow\infty$ controlled $\left(
\varepsilon^{k},T^{k}\right)  $-chains $\zeta^{k}$ from $(\omega,x)$ to
$(\overline{\omega},y)$. Let the final segments of the chains $\zeta^{k}$ be
given by $(\omega_{n^{k}-1}^{k}\cdot t,\varphi(t,\omega_{n^{k}-1}^{k}%
,x_{n^{k}-1}^{k},u_{n^{k}-1}^{k})),t\in\lbrack0,T_{n^{k}-1}^{k}]$. Without
loss of generality, $u_{n^{k}-1}^{k}(T_{n^{k}-1}^{k}+\cdot)$ converges to some
control $v\in\mathcal{U}$ and $\omega_{n^{k}-1}^{k}$ converges to some
$\omega^{\prime}\in\Omega$. For $k\rightarrow\infty$, it follows that
\begin{align*}
\left.  d(\omega_{n^{k}}^{k},\overline{\omega})=\right.  d(\omega_{n^{k}%
-1}^{k}\cdot T_{n^{k}-1}^{k},\overline{\omega})  &  <\varepsilon
^{k}\rightarrow0,\\
d\bigl(\varphi(T_{n^{k}-1}^{k},\omega_{n^{k}-1}^{k},x_{n^{k}-1}^{k}%
,u_{n^{k}-1}^{k}),y\bigr)  &  <\varepsilon^{k}\rightarrow0.
\end{align*}
It follows that%
\begin{align*}
&  \varphi(\tau+T_{n^{k}-1}^{k},\omega_{n^{k}-1}^{k},x_{n^{k}-1}^{k}%
,u_{n^{k}-1}^{k})\\
&  =\varphi\bigl(\tau,\omega_{n^{k}-1}^{k}\cdot T_{n^{k}-1}^{k},\varphi
(T_{n^{k}-1}^{k},\omega_{n^{k}-1}^{k},x_{n^{k}-1}^{k},u_{n^{k}-1}%
^{k}),u_{n^{k}-1}^{k}\left(  T_{n^{k}-1}^{k}+\cdot\right)  \bigr)\\
&  \rightarrow\varphi(\tau,\overline{\omega},y,v).
\end{align*}
We claim that $(\overline{\omega}\cdot\tau,\varphi(\tau,y,\overline{\omega
},v))\in\mathbf{R}^{c}(\omega,x)$. For the proof, we construct for all
$\varepsilon,T>0$ controlled $(\varepsilon,T)$-chains from $(\omega,x)$ to
$(\overline{\omega}\cdot\tau,\varphi(\tau,\overline{\omega},y,v))$. Let
$\varepsilon,T>0$. For $k$ large enough, the times satisfy $T_{j}^{k}\geq T$
and $\tau+T_{n^{k}-1}^{k}\geq T$. There is $\delta>0$ such that, for
$(\widetilde{\omega},\widetilde{y},\widetilde{u})\in\Omega\times
M\times\mathcal{U}$,%
\[
d(\widetilde{\omega},\overline{\omega})<\delta,d(\widetilde{y},y)<\delta
,d(\widetilde{u},v)<\delta
\]
implies%
\[
d(\widetilde{\omega}\cdot\tau,\overline{\omega}\cdot\tau)<\varepsilon\text{
and }d\left(  \varphi(\tau,\widetilde{\omega},\widetilde{y},\widetilde
{u}),\varphi(\tau,\overline{\omega},y,v)\right)  <\varepsilon.
\]
For $k$ large enough, it holds that $\varepsilon^{k}<\delta$, and hence
$d(\omega_{n^{k}-1}^{k}\cdot T_{n^{k}-1}^{k},\overline{\omega})<\delta$ and%
\[
d\left(  \varphi(T_{n^{k}-1}^{k},\omega_{n^{k}-1}^{k},x_{n^{k}-1}^{k}%
,u_{n^{k}-1}^{k}),y\right)  <\delta,d(u_{n^{k}-1}^{k}\left(  T_{n^{k}-1}%
^{k}+\cdot\right)  ,v)<\delta.
\]
It follows that $d(\omega_{n^{k}-1}^{k}\cdot(T_{n^{k}-1}^{k}+\tau
),\overline{\omega}\cdot\tau)<\varepsilon$ and%
\begin{align*}
&  d\left(  \varphi(\tau+T_{n^{k}-1}^{k},\omega_{n^{k}-1}^{k},x_{n^{k}-1}%
^{k},u_{n^{k}-1}^{k}),\varphi(\tau,\overline{\omega},y,v)\right) \\
&  =d\bigl(\varphi(\tau,\omega_{n^{k}-1}^{k}\cdot T_{n^{k}-1}^{k}%
,\varphi(T_{n^{k}-1}^{k},\omega_{n^{k}-1}^{k},x_{n^{k}-1}^{k},u_{n^{k}-1}%
^{k}),u_{n^{k}-1}^{k}\left(  T_{n^{k}-1}^{k}+\cdot\right)  ),\\
&  \left.  \qquad\qquad\qquad\varphi(\tau,\overline{\omega}%
,y,v)\bigr)<\varepsilon\right.  .
\end{align*}
Thus we may replace the final segment of the controlled $\left(
\varepsilon^{k},T^{k}\right)  $-chain $\zeta^{k}$ by%
\[
(\omega_{n^{k}-1}^{k}\cdot t,\varphi(t,x_{n^{k}-1}^{k},\omega_{n^{k}-1}%
^{k},u_{n^{k}-1}^{k})),t\in\lbrack0,\tau+T_{n^{k}-1}^{k}]
\]
and obtain the desired controlled $(\varepsilon,T)$-chain from $(\omega,x)$ to
$(\overline{\omega}\cdot\tau,\varphi(\tau,y,\overline{\omega},v))$.

(ii) Let $(\omega,x)$ be in the maximal chain controllable set $F$. First we
show that the control $v$ constructed above (defined only for negative time)
satisfies, for $\tau<0$%
\begin{equation}
(\omega\cdot\tau,\varphi(\tau,\omega,x,v))\in F. \label{F1}%
\end{equation}
By definition of $F$ it holds that $(\omega,x)\in\mathbf{R}^{c}(\overline
{\omega},y)$ for every $(\overline{\omega},y)\in F$. As shown in part (i) of
the proof, it follows that $\left(  \omega\cdot\tau,\varphi(\tau
,\omega,x,v)\right)  \in\mathbf{R}^{c}(\overline{\omega},y)$. We claim that,
for every $(\overline{\omega},y)\in F$,
\[
(\overline{\omega},y)\in\mathbf{R}^{c}(\omega\cdot\tau,\varphi(\tau
,\omega,x,v)).
\]
Since $F$ is a maximal chain controllable set, the claim implies (\ref{F1})
for $\tau<0$.

In order to prove the claim, consider a controlled $(\varepsilon,T)$-chain
from $(\omega,x)$ to $(\overline{\omega},y)$. Modify the first segment
$(\omega\cdot t,\varphi(t,\omega,x,u_{0})),t\in\lbrack0,T_{0}]$, in the
following way: define a control%
\[
u_{0}^{\prime}(t)=\left\{
\begin{array}
[c]{lll}%
v(t+\tau) & \text{for} & t\in\lbrack0,-\tau]\\
u_{0}(t-\tau) & \text{for} & t\in(-\tau,-\tau+T_{0}]
\end{array}
\right.
\]
and let the modified segment be%
\[
\left(  \left(  \omega\cdot\tau\right)  \cdot t,\varphi(t,\omega\cdot
\tau,\varphi(\tau,\omega,x,u),u_{0}^{\prime})\right)  ,t\in\lbrack
0,-\tau+T_{0}].
\]
Together with the other segments this yields a controlled $(\varepsilon
,T)$-chain from $(\omega\cdot\tau,\varphi(\tau,\omega,x,u))$ to $(\overline
{\omega},y)$. \ This completes the proof of the claim.

It remains to consider the case of positive $\tau$. We have to construct a
control $v\in\mathcal{U}$ such that (\ref{F1}) holds for $\tau>0$.

Let $(\overline{\omega},y)\in F$. Then there exists a controlled
$(\varepsilon,T)$-chain from $(\overline{\omega},y)$ to $(\omega,x)$. As in
the proof of (i), for any control $u\in\mathcal{U}$, one can construct
modified controlled $(\varepsilon,T)$-chains from $(\overline{\omega},y)$ to
$(\omega\cdot\tau,\varphi(\tau,\omega,x,u))$. Hence it remains to construct a
control $v\in\mathcal{U}$ such that, for $\tau>0$ and all $\varepsilon,T>0$
there are controlled $(\varepsilon,T)$-chains from $(\omega\cdot\tau
,\varphi(\tau,\omega,x,v))$ to $(\overline{\omega},y)$.

For sequences $\varepsilon^{k}\rightarrow0$ and $T^{k}\rightarrow\infty$,
consider controlled $\left(  \varepsilon^{k},T^{k}\right)  $-chains $\zeta
^{k}$ from $(\omega,x)$ to $(\overline{\omega},y)$. Let the initial segments
of the chains $\zeta^{k}$ be given by $(\omega\cdot t,\varphi(t,\omega
,x,u_{0}^{k})),t\in\lbrack0,T_{0}^{k}]$. We may assume that the controls
$u_{0}^{k}\in\mathcal{U}$ converge to some $v\in\mathcal{U}$ defined on
$[0,\infty)$, and hence%
\[
\varphi(\tau,\omega,x,u_{0}^{k})\rightarrow\varphi(\tau,\omega,x,v)\text{ for
}\tau>0.
\]
Let $\varepsilon,T>0$. For $k$ large enough, one obtains $\varepsilon
^{k}<\varepsilon/2$ and $T_{0}^{k}-\tau\geq T$. In order to construct an
$(\varepsilon,T)$-chain, we may, if necessary, introduce a trivial jump
replacing $T_{0}^{k}$ by a time $\widetilde{T}_{0}^{k}$ with $\widetilde
{T}_{0}^{k}-\tau\in\lbrack T,2T]$; cf. Remark \ref{Remark_2T}. Using
compactness of $\mathcal{U}$ and continuity of $\varphi$, one finds $\delta>0$
such that $d(\varphi(\tau,\omega,x,v),\widetilde{y})<\delta$ implies for all
$u\in\mathcal{U}$
\[
d\left(  \varphi(t,\omega\cdot\tau,\varphi(\tau,\omega,x,v),u),\varphi
(t,\omega\cdot\tau,\widetilde{y},u)\right)  <\varepsilon/2\text{ for all }%
t\in\lbrack T,2T]\text{ and all }u\in\mathcal{U}.
\]
It follows that, for $k$ large enough,%
\begin{align*}
&  \varphi(\widetilde{T}_{0}^{k}-\tau,\omega\cdot\tau,\varphi(\tau
,\omega,x,v),u_{0}^{k}(\tau+\cdot)),x_{1}^{k})\\
&  \leq\varphi(\widetilde{T}_{0}^{k}-\tau,\omega\cdot\tau,\varphi(\tau
,\omega,x,v),u_{0}^{k}(\tau+\cdot)),\varphi(\widetilde{T}_{0}^{k}-\tau
,\omega\cdot\tau,\varphi(\tau,\omega,x,u_{0}^{k}),u_{0}^{k}(\tau+\cdot)))\\
&  \qquad+d(\varphi(\widetilde{T}_{0}^{k},\omega,x,u_{0}^{k}),x_{1}^{k})\\
&  <\varepsilon/2+\varepsilon^{k}<\varepsilon.
\end{align*}
Thus, we can replace in the controlled $(\varepsilon^{k},T^{k})$-chain
$\zeta^{k}$ the initial segment by%
\[
\varphi(t,\omega\cdot\tau,\varphi(\tau,\omega,x,v),u_{0}^{k}(\tau+\cdot
)),t\in\lbrack0,\widetilde{T}_{0}^{k}-\tau],
\]
and obtain a controlled $(\varepsilon,T)$-chain from $(\omega\cdot\tau
,\varphi(\tau,\omega,x,v))$ to $(\overline{\omega},y)\,$. This completes the proof.
\end{proof}

An immediate consequence of Proposition \ref{Proposition_inv}(ii) is the
following result.

\begin{proposition}
\label{Proposition_contained2}Suppose that $F\subset\Omega\times M$ is a chain
controllable set, which is contained in a maximal chain controllable set in
$\Omega\times Q$ with $Q\subset M$ compact. Then $F$ is contained in a chain
control set.
\end{proposition}

It is of interest to see if the behavior in a single fiber determines chain
control sets. In the periodic case, one can reconstruct chain control sets
from their intersection with a fiber; cf. Gayer \cite{Gaye05}. We will prove a
weaker property for general nonautonomous systems. In the following theorem we
suppose that the jump times $T_{i}$ of the involved controlled $(\varepsilon
,T)$-chains satisfy $T_{i}\in\lbrack T,2T]$; cf. Remark \ref{Remark_2T}.

\begin{theorem}
\label{Theorem_single_fiber}Consider control system (\ref{control}). Fix
$\omega_{0}\in\Omega$ and suppose that there is a nonvoid set $F_{\omega_{0}%
}\subset M$ with the property that for all $x_{0},y_{0}\in F_{\omega_{0}}$ and
all $\varepsilon,T>0$ there exists a controlled $(\varepsilon,T)$-chain from
$(\omega_{0},x_{0})$ to $(\omega_{0},y_{0})$. Then the set
\[
F:=\left\{  (\omega,z)\in\Omega\times M\left\vert
\begin{array}
[c]{c}%
\text{$\exists x_{0},y_{0}\in F_{\omega_{0}}~$}\forall\text{$\varepsilon
,T>0~$}\exists\text{ controlled $(\varepsilon,3T)$-chain}\\
\text{from $(\omega_{0},x_{0})$ to $(\omega_{0},y_{0})~\exists$}%
i\in\text{$\{0,\dots,n-1\}$}\\
\exists\tau\in\lbrack T,2T]:\omega=\omega_{i}\cdot\tau,d(\text{$z,\varphi
(\tau,\omega_{i},x_{i},u_{i}))$}<\varepsilon
\end{array}
\right.  \right\}
\]
is chain controllable. The set $F$ is contained in a chain control set, if $F$
is contained in a maximal chain controllable set in $\Omega\times Q$ with
$Q\subset M$ compact.
\end{theorem}

\begin{proof}
By Proposition \ref{Proposition_contained2}, it suffices to show that $F$ is
chain controllable. For $\omega\in\Omega$, write $F_{\omega}:=\left\{  z\in
M\left\vert (\omega,z)\in F\right.  \right\}  $. Fix $(\omega,z),(\omega
^{\prime},z^{\prime})\in F$ and $\varepsilon,T>0$. Continuity of $\varphi$ and
compactness of $[0,6T]\times\Omega\times\mathcal{U}$ implies that there is
$\delta\in(0,\varepsilon)$ such that $d(y,z)<\delta$ implies for all
$t\in\lbrack0,6T],\omega\in\Omega$, and $u\in\mathcal{U}$%
\begin{equation}
d\left(  \varphi(t,\omega,y,u),\varphi(t,\omega,z,u)\right)  <\varepsilon/2.
\label{cont_z}%
\end{equation}
There exist $x_{0},y_{0}\in F_{\omega_{0}}$ and a controlled $(\delta
,3T)$-chain from $(\omega_{0},x_{0})$ to $(\omega_{0},y_{0})$ with%
\[
\omega=\omega_{i}\cdot\tau\text{ and }d\left(  z,\varphi(\tau,\omega_{i}%
,x_{i},u_{i})\right)  <\delta
\]
for some $i$ and $\tau\in\lbrack T,2T]$. Then $\tau+T\leq3T\leq T_{i}\leq6T$,
hence $T_{i}-\tau\in\lbrack T,5T]$, and, by (\ref{cont_z}),%
\begin{align*}
&  d\left(  \varphi(T_{i}-\tau,\omega_{i}\cdot\tau,z,u_{i}(\tau+\cdot
)),\varphi(T_{i},\omega_{i},x_{i},u_{i})\right) \\
&  =d\left(  \varphi(T_{i}-\tau,\omega_{i}\cdot\tau,z,u_{i}(\tau
+\cdot)),\varphi(T_{i}-\tau,\omega_{i}\cdot\tau,\varphi(\tau,\omega_{i}%
,x_{i},u_{i}),u_{i}(\tau+\cdot)\right)  <\varepsilon.
\end{align*}
It follows that the second part of this chain defines a controlled
$(\varepsilon,T)$-chain $\zeta^{1}$ from $(\omega,z)$ to $(\omega_{0},y_{0})$.
It remains to construct a controlled $(\varepsilon,T)$-chain $\zeta^{2}$ from
$(\omega_{0},y_{0})$ to $(\omega^{\prime},z^{\prime})$. Since $(\omega
^{\prime},z^{\prime})\in F$, there exist $x_{0}^{\prime},y_{0}^{\prime}\in
F_{\omega_{0}}$ and a controlled $(\varepsilon,3T)$-chain from $(\omega
_{0},x_{0}^{\prime})$ to $(\omega_{0},y_{0}^{\prime})$ with
\begin{equation}
\text{ }\omega^{\prime}=\omega_{j}^{\prime}\cdot\tau^{\prime}\text{ and
}d\left(  z^{\prime},\varphi(\tau^{\prime},\omega_{j}^{\prime},x_{j}^{\prime
},u_{j}^{\prime})\right)  <\varepsilon\label{end}%
\end{equation}
for some $j$ and $\tau^{\prime}\in\lbrack T,2T]$. We modify the first part of
this chain so that it becomes an $(\varepsilon,T)$-chain from $(\omega
_{0},x_{0}^{\prime})$ to $(\omega^{\prime},z^{\prime})$:\ Instead of the
segment $\varphi(t,\omega_{j}^{\prime},x_{j}^{\prime},u_{j}^{\prime}%
),t\in\lbrack0,T_{j}^{\prime}]$, consider the segment $\varphi(t,\omega
_{j}^{\prime},x_{j}^{\prime},u_{j}^{\prime}),t\in\lbrack0,\tau^{\prime}]$. By
(\ref{end}), this defines a controlled $(\varepsilon,T)$-chain $\zeta^{3}$
from $(\omega_{0},x_{0}^{\prime})$ to $(\omega^{\prime},z^{\prime})$.

Finally, there exists a controlled $(\varepsilon,T)$-chain $\zeta^{2}$ from
$(\omega_{0},y_{0})$ to $(\omega_{0},x_{0}^{\prime})$. We have constructed
$(\varepsilon,T)$-chains $\zeta^{1}$ from $(\omega,z)$ to $(\omega_{0},y_{0}%
)$, $\zeta^{2}$ from $(\omega_{0},y_{0})$ to $(\omega_{0},x_{0}^{\prime})$ and
$\zeta^{3}$ from $(\omega_{0},x_{0}^{\prime})$ to $(\omega^{\prime},z^{\prime
})$. Since $\varepsilon,T>0$ are arbitrary and $x_{0}^{\prime}$ is independent
of $\varepsilon,T$, it follows from Lemma \ref{Lemma_concatenation}, applied
twice, that for all $\varepsilon,T>0$ there are controlled $(\varepsilon
,T)$-chains from $(\omega,z)$ to $(\omega^{\prime},z^{\prime})$ proving the claim.
\end{proof}

\begin{remark}
Theorem~\ref{Theorem_single_fiber} shows that one can find chain control sets
by looking at a single fiber, i.\thinspace e.,\ a single excitation. This
significantly simplifies numerical computations, since only one excitation
$\omega\cdot t,t\in\mathbb{R}$, has to be considered (cf. Colonius and
Wichtrey \cite[Section 7]{ColW09}). Theorem~\ref{Theorem_single_fiber} and its
proof generalize and correct \cite[Proposition 3.6]{ColW09}, where almost
periodic excitations $\omega$ were considered.
\end{remark}

\section{Relation to chain transitive sets\label{transitivity}}

In this section, we relate chain control sets to dynamical objects of the skew
product flow $\Phi$, which is a nonautonomous control flow. In the autonomous
case, chain control sets are the projections of maximal chain transitive sets
for the control flow. In the nonautonomous setting here, no jumps in $\Omega$
are allowed. Hence we also have to define the dynamical objects for the
control flow accordingly.

\begin{definition}
Let $(u,\omega,x),(v,\overline{\omega},y)\in\mathcal{U}\times\Omega\times M$
and fix $\varepsilon,T>0$. For the (nonautonomous) control flow $\Phi$ on
$\mathcal{U}\times\Omega\times M$ an $(\varepsilon,T)$-chain $\zeta$ from
$(u,\omega,x)$ to $(v,\overline{\omega},y)$ is given by $n\in\mathbb{N}$,
elements $(u_{0},\omega_{0},x_{0})=(u,\omega,x),(u_{1},\omega_{1},x_{1}%
),\dots,\allowbreak(u_{n},\omega_{n},x_{n})=(v,\omega_{n},y)\in
\mathcal{U\times}\Omega\times M$, and times $T_{0},\dots,\allowbreak
T_{n-1}\geq T$ such that

(i) $\omega_{j}\cdot T_{j}=\omega_{j+1}$ for $j=0,\dots,n-1$, and
$d(\omega_{n},\overline{\omega})<\varepsilon$,

(ii) $d\bigl((\theta_{T_{j}}u_{j},\varphi(T_{j},\omega_{j},x_{j}%
,u_{j})),(u_{j+1},x_{j+1})\bigr)<\varepsilon$ for $j=0,\dots,n-1$.
\end{definition}

\begin{definition}
A chain transitive set $\mathcal{E}$ for the control flow $\Phi$ is a subset
of $\mathcal{U}\times\Omega\times M$ such that for all $(u,\omega
,x),(v,\overline{\omega},y)\in\mathcal{E}$ and all $\varepsilon,T>0$ there is
an $(\varepsilon,T)$-chain $\zeta$ from $(u,\omega,x)$ to $(v,\overline
{\omega},y)$.
\end{definition}

If for all $\varepsilon,T>0$ all segments $\varphi(t,\omega_{j},x_{j}%
,u_{j}),t\in\lbrack0,T_{j}]$, of the $(\varepsilon,T)$-chains are contained in
a subset $Q\subset M$, we say that $\mathcal{E}$ is a chain transitive set in
$\mathcal{U}\times\Omega\times Q$. We emphasize that chain transitivity for
nonautonomous flows $\Phi$, as defined above, does not coincide with chain
transitivity of the flow $\Phi$ on $\mathcal{U}\times\Omega\times M$, since no
jumps in $\Omega$ are allowed.

\begin{remark}
Chen and Duan \cite[Definition 2.4]{ChenDu11} define chain transitivity for
nonautonomous dynamical systems on noncompact spaces using the pullback
concept. Their main result \cite[Theorem 1.1]{ChenDu11} is a decomposition of
the state space into a chain recurrent part and a gradient-like part. Observe
also that $\mathcal{E}$ is a nonautonomous set in the sense of Kloeden and
Rasmussen \cite[Definition 3.2]{KloR11} for $\Phi$ considered as a skew
product flow with base space $\Omega$ as indicated in (\ref{skew3}).
\end{remark}

Note the following property.

\begin{lemma}
\label{Lemma_concatenation2}Consider $(u^{i},\omega^{i},x^{i})\in
\mathcal{U}\times\Omega\times M,i=1,2,3$, and assume that for all
$\varepsilon,T>0$ there are $(\varepsilon,T)$-chains from $(u^{1},\omega
^{1},x^{1})$ to $(u^{2},\omega^{2},x^{2})$ and from $(u^{2},\omega^{2},x^{2})$
to $(u^{3},\omega^{3},x^{3})$. Then, for all $\varepsilon,T>0$, there are
$(\varepsilon,T)$-chains from $(u^{1},\omega^{1},x^{1})$ to $(u^{3},\omega
^{3},x^{3})$.
\end{lemma}

\begin{proof}
The proof is analogous to the proof of Lemma \ref{Lemma_concatenation}, and
hence we omit it.
\end{proof}

We also note the following concept.

\begin{definition}
Consider the nonautonomous control flow $\Phi$ on $\mathcal{U}\times
\Omega\times M$. The forward chain limit set for $\Phi$ is
\[
\Omega^{+}(u,\omega,x):=\left\{  (v,\omega^{\prime},y)\in\mathcal{U}%
\times\Omega\times M\left\vert
\begin{array}
[c]{c}%
\forall\varepsilon,T>0\,\exists(\varepsilon,T)\text{-chain}\\
\text{ from }\left(  u,\omega,x\right)  \text{ to }(v,\omega^{\prime},y)
\end{array}
\right.  \right\}  .
\]

\end{definition}

\begin{proposition}
\label{Proposition_flow_inv}A maximal chain transitive set $Y$ of $\Phi$ in
$\mathcal{U}\times\Omega\times Q$, where $Q\subset M$ is compact, is invariant.
\end{proposition}

\begin{proof}
The chain transitive set $Y$ is invariant if $(u,\omega,x)\in Y$ implies that
$\Phi_{\tau}(u,\omega,x)\in Y$ for all $\tau\in\mathbb{R}$. Thus we have to
show that for $\tau\in\mathbb{R}$ and $(v,\overline{\omega},y)\in Y$ it
follows that
\[
\Phi_{\tau}(u,\omega,x)\in\Omega^{+}(v,\overline{\omega},y)\text{ and
}(v,\overline{\omega},y)\in\Omega^{+}(\Phi_{\tau}(u,\omega,x)).
\]
(i) First we prove that $\Phi_{\tau}(u,\omega,x)\in\Omega^{+}(v,\overline
{\omega},y)$. Let $\varepsilon,T>0$. By continuity there is $\delta>0$ such
that
\begin{equation}
d((u^{\prime},\omega^{\prime},x^{\prime}),(u,\omega,x))<\delta\text{ implies
}d\left(  \Phi_{\tau}(u^{\prime},\omega^{\prime},x^{\prime}),\Phi_{\tau
}(u,\omega,x)\right)  <\varepsilon. \label{uniform3}%
\end{equation}
Pick a $(\delta,T^{\prime})$-chain from $(v,\overline{\omega},y)$ to
$(u,\omega,x)$ with $T^{\prime}=T+\left\vert \tau\right\vert $, and hence
$T_{n-1}+\tau\geq T$. Then $d(\omega_{n},\omega)<\delta$ and%
\[
d\bigl((\theta_{T_{n-1}}u_{n-1},\varphi(T_{n-1},\omega_{n-1},x_{n-1}%
,u_{n-1})),(u,x)\bigr)<\delta.
\]
Hence $d(\omega_{n}\cdot\tau,\omega\cdot\tau)<\varepsilon$ and
\begin{align*}
&  d\left(  \Phi_{\tau+T_{n-1}}(u_{n-1},\omega_{n-1},x_{n-1}),\Phi_{\tau
}(u,\omega_{n},x)\right) \\
&  =d\left(  \Phi_{\tau}(\Phi_{T_{n-1}}(u_{n-1},\omega_{n-1},x_{n-1}%
)),\Phi_{\tau}(u,\omega,x)\right)  <\varepsilon.
\end{align*}
This yields an $(\varepsilon,T)$-chain from $(v,\overline{\omega},y)$ to
$\Phi_{\tau}(u,\omega,x)$ showing that $\Phi_{\tau}(u,\omega,x)\in\Omega
^{+}(v,\overline{\omega},y)$.

(ii) Let $\tau\in\mathbb{R}$ and $(u,\omega,x)\in Y$. We claim that
\[
(v,\overline{\omega},y)\in\Omega^{+}(\Phi_{\tau}(u,\omega,x)).
\]
By (i) it follows, for all $(v,\overline{\omega},y)\in Y$ that $\Phi_{\tau
}(u,\omega,x)\in\Omega^{+}(v,\overline{\omega},y)$. Furthermore, it follows
from $(v,\overline{\omega},y)\in\Omega^{+}(u,\omega,x)$ that $\Phi_{-\tau
}(v,\overline{\omega},y)\in\Omega^{+}(u,\omega,x)$. Here we use the
compactness assumption for $Q$: For $\varepsilon>0$ there is $\delta>0$ such
that $d(y^{\prime},y^{\prime\prime})<\delta$ in $Q$ and $d(\omega^{\prime
},\omega^{\prime\prime})<\delta$ implies, for all $u\in\mathcal{U}$,%
\begin{equation}
d(\omega^{\prime}\cdot\tau,\omega^{\prime\prime}\cdot\tau)<\varepsilon\text{
and }d\left(  \varphi(\tau,\omega^{\prime},y^{\prime},u),\varphi(\tau
,\omega^{\prime\prime},y^{\prime\prime},u)\right)  <\varepsilon.
\label{uniform4}%
\end{equation}
There is a $(\delta,T)$-chain $\zeta$ in $\mathcal{U}\times\Omega\times Q$
from $(u,\omega,x)$ to $\Phi_{-\tau}(v,\overline{\omega},y)$ given by
$n\in\mathbb{N}$ and $(u_{0},\omega_{0},x_{0})=(u,\omega,x),\dots
,(u_{n},\omega_{n},x_{n})=(\theta_{-\tau}v,\omega_{n},\varphi(-\tau
,\overline{\omega},y,\theta_{-\tau}v)$ in $\mathcal{U\times}\Omega\times Q$,
and times $T_{0},\dots,T_{n-1}\geq T$ such that

(i) $\omega_{j}\cdot T_{j}=\omega_{j+1}$ for $j=0,\dots,n-1$, and
$d(\omega_{n},\overline{\omega}\cdot(-\tau))<\delta$,

(ii) $d\bigl((\theta_{T_{j}}u_{j},\varphi(T_{j},\omega_{j},x_{j}%
,u_{j})),(u_{j+1},x_{j+1})\bigr)<\varepsilon$ for $j=0,\dots,n-1$.

This gives rise to the following $(\varepsilon,T)$-chain from $\Phi_{\tau
}(u,\omega,x)$ to $(v,\overline{\omega},y)$.\ Let the times be given by
$T_{0},\dots,T_{n-1}$ and let $(u_{0}^{\prime},\omega_{0}^{\prime}%
,x_{0}^{\prime})=\Phi_{\tau}(u,\omega,x)$,%
\begin{align*}
(u_{j}^{\prime},\omega_{j}^{\prime},x_{j}^{\prime})  &  =(\theta_{\tau}%
u_{j},\omega_{j}\cdot\tau,\varphi(\tau,\omega_{j},x_{j},u_{j}))\text{ for
}j=0,\ldots,n-1,\\
(u_{n}^{\prime},\omega_{n}^{\prime},x_{n}^{\prime})  &  =(\theta_{\tau}%
u_{n},\omega_{n}\cdot\tau,\varphi(\tau,\omega_{n},x_{n},u_{n}))=(v,\omega
_{n}\cdot\tau,y)\in\mathcal{U\times}\Omega\times M.
\end{align*}
Using (\ref{uniform4}) we verify that

(i) $\omega_{j}^{\prime}\cdot T_{j}=\omega_{j}\cdot(\tau+T_{j})=\omega
_{j+1}\cdot\tau=\omega_{j+1}^{\prime}$ for $j=0,\dots,n-1$, and $d(\omega
_{n}^{\prime},(\overline{\omega}\cdot(-\tau))\cdot\tau))=d(\omega_{n}\cdot
\tau,\overline{\omega})<\varepsilon$,

(ii) $d\bigl((\theta_{T_{j}}\theta_{\tau}u_{j},\varphi(\tau+T_{j},\omega
_{j},x_{j},u_{j})),(\theta_{\tau}u_{j+1},\varphi(\tau,\omega_{j+1}%
,x_{j+1},u_{j+1})\bigr)<\varepsilon$ for $j=0,\dots,n-1$.

It follows that $(v,\overline{\omega},y)\in\Omega^{+}(\Phi_{\tau}%
(u,\omega,x))$, as claimed. This completes the proof of the proposition.
\end{proof}

We cite the following lemma (cf. Colonius and Kliemann \cite[Lemma
4.2.1]{ColK00} or Kawan \cite[Proposition 1.14]{Kawan13}).

\begin{lemma}
\label{Lemma_metric}The set $\mathcal{U}$ is compact and metrizable in the
weak$^{\ast}$ topology of $L^{\infty}(\mathbb{R}\mathbf{,}\mathbb{R}%
^{m})=(L^{1}(\mathbb{R}\mathbf{,}\mathbb{R}^{m}))^{\ast}$; a metric is given
by
\begin{equation}
d(u,v)=\sum_{i=1}^{\infty}\frac{1}{2^{i}}\frac{\mid\int_{\mathbb{R}}\langle
u(t)-v(t),y_{i}(t)\rangle\,dt\mid}{1+\mid\int_{\mathbb{R}}\langle
u(t)-v(t),y_{i}(t)\rangle\,dt\mid}, \label{metric on U}%
\end{equation}
where $\left\{  y_{i},\;i\in\mathbb{N}\right\}  $ is a countable, dense subset
of $L^{1}(\mathbb{R}\mathbf{,}\mathbb{R}^{m})$, and $\langle\cdot$%
,$\cdot\rangle$\thinspace denotes an inner product in\thinspace\thinspace
$\mathbb{R}^{m}.$ With this metric, $\mathcal{U}$ is a compact, complete,
separable metric space.
\end{lemma}

The following theorem establishes the equivalence of chain control sets and
maximal invariant chain transitive sets for the control flow.

\begin{theorem}
\label{Theorem_equivalence}Consider the nonautonomous control system given by
(\ref{control}).

(i) If $E\subset\Omega\times M$ is a chain control set, then the lift%
\[
\mathcal{E}:=\{(u,\omega,x)\in\mathcal{U}\times\Omega\times M\left\vert
\forall t\in\mathbb{R}:\psi(t,\omega,x,u)\in E\right.  \}
\]
is a maximal invariant chain transitive set for the control flow $\Phi$.

(ii) Conversely, let $\mathcal{E}\subset\mathcal{U}\times\Omega\times M$ be a
maximal invariant chain transitive set for $\Phi$. Then the projection to
$\Omega\times M$,%
\[
\pi_{\Omega\times M}\mathcal{E}:=\left\{  (\omega,x)\in\Omega\times
M\left\vert \exists u\in\mathcal{U}:(u,\omega,x)\in\mathcal{E}\right.
\right\}
\]
is a chain control set.
\end{theorem}

\begin{proof}
(i) It is clear that the lift $\mathcal{E}$ is invariant. We show that
$\mathcal{E}$ is chain transitive. Let $(u,\omega,x),(v,\overline{\omega
},y)\in\mathcal{E}$ and pick $\varepsilon,\,T>0.$ Recall the definition of the
metric $d$ on $\mathcal{U}$ given in (\ref{metric on U}) and choose
$N\in\mathbb{N}$ large enough such that $\sum_{i=N+1}^{\infty}2^{-i}%
<\frac{\varepsilon}{2}$. For the finitely many $y_{1},\dots,y_{N}\in
L^{1}(\mathbb{R},\mathbb{R}^{m}),$ there exists $S>0$ such that for all $i$%
\[
\int_{\mathbb{R}\setminus\left[  -S,S\right]  }\left\vert y_{i}(\tau
)\right\vert \,\,d\tau<\frac{\varepsilon}{2\mathrm{\,diam}\,U}.
\]
Without loss of generality, we can assume that $T\geq S$. There is $\delta
\in(0,\varepsilon)$ such that%
\[
d(\omega^{\prime},\overline{\omega}\cdot(-T))<\delta\text{ implies }%
d(\omega^{\prime}\cdot T,\overline{\omega})<\varepsilon.
\]
For the chain control set $E$, there exists a controlled $(\delta,T)$-chain
from $\psi(2T,\omega,x,u)\allowbreak\in E$ to $\psi(-T,y,\overline{\omega
},v)\in E$, and hence there are $n\in\mathbb{N}$ and $x_{0},\dots,x_{n}\in
M,\,u_{0},\dots,\allowbreak u_{n-1}\allowbreak\in\mathcal{U},\,T_{0}%
,\dots,T_{n-1}\geq T$ with $\omega_{j}\cdot T_{j}=\omega_{j+1},d(\omega
_{n},\overline{\omega}\cdot(-T))<\delta$, and%
\begin{align*}
(\omega\cdot(2T),x_{0})  &  =\psi(2T,\omega,x,u),(\omega_{n},x_{n}%
)=(\omega_{n},y),\\
d(\varphi(T_{j},\omega_{j},x_{j},u_{j}),x_{j+1})  &  <\delta\text{\ for
}j=0,\dots,n-1.
\end{align*}
Since $d(\omega_{n},\overline{\omega}\cdot(-T))<\delta$ the choice of $\delta$
implies that $d(\omega_{n}\cdot T,\overline{\omega})<\varepsilon$. We now
construct an $(\varepsilon,T)$-chain from $(u,\omega,x)$ to $(v,\overline
{\omega},y)$ in the following way. Define
\[%
\begin{array}
[c]{cccc}%
T_{-2}=T, & x_{-2}=x, & v_{-2}=u, & \\
T_{-1}=T, & x_{-1}=\varphi(T,x,u), & v_{-1}(t)= & \left\{
\begin{array}
[c]{ll}%
u(T_{-2}+t) & \text{for }t\leq T_{-1}\\
u_{0}(t-T_{-1}) & \text{for }t>T_{-1}%
\end{array}
\right.
\end{array}
\]
and let the times $T_{0},\dots,T_{n-1}$ and the points $x_{0},\dots,x_{n}$ be
as given earlier; furthermore, set
\[%
\begin{array}
[c]{ccc}%
T_{n}=T, & x_{n+1}=y, & v_{n+1}=v,
\end{array}
\]
and define, for $j=0,\dots,n-2$, controls by%
\begin{align*}
v_{j}(t)  &  =\left\{
\begin{array}
[c]{lll}%
v_{j-1}(T_{j-1}+t) & \text{for} & t\leq0\\
u_{j}(t) & \text{for} & 0<t<T_{j}\\
u_{j+1}(t-T_{j}) & \text{for} & t>T_{j},
\end{array}
\right. \\
v_{n-1}(t)  &  =\left\{
\begin{array}
[c]{lll}%
v_{n-2}(T_{n-2}+t) & \text{for} & t\leq0\\
u_{n-1}(t) & \text{for} & 0<t\leq T_{n-1}\\
v(t-T_{n-1}-T) & \text{for} & t>T_{n-1},
\end{array}
\right. \\
v_{n}(t)  &  =\left\{
\begin{array}
[c]{ll}%
v_{n-1}(T_{n-1}+t) & \text{for }t\leq0\\
v(t-T) & \text{for }t>0.
\end{array}
\right.
\end{align*}
Since $d(\omega_{n}\cdot T,\overline{\omega})<\varepsilon$ it follows that%
\[
(v_{-2},\omega,x_{-2}),\,(v_{-1},\omega\cdot T,x_{-1}),\dots,(v_{n+1}%
,\omega_{n}\cdot T,x_{n+1})\text{ and }T_{-2},\,T_{-1},\dots,T_{n}\geq T,
\]
constitute an $(\varepsilon,T)$-chain from $(u,\omega,x)$ to $(v,\overline
{\omega},y)$ provided that for $j=-2,\,-1,\dots,n$%
\[
d(v_{j}(T_{j}+\cdot),v_{j+1})<\varepsilon.
\]
By choice of $T\geq S$ and $N$, one has, for all $w_{1},\,w_{2}\in\mathcal{U}%
$,%
\begin{align}
d(w_{1},w_{2})  &  =\sum_{i=1}^{\infty}2^{-i}\frac{\left\vert \int
_{\mathbb{R}}\left\langle w_{1}(t)-w_{2}(t),y_{i}(t)\right\rangle
\,dt\right\vert }{1+\left\vert \int_{\mathbb{R}}\left\langle w_{1}%
(t)-w_{2}(t),y_{i}(t)\right\rangle \,dt\right\vert }\label{use_metric}\\
&  \leq\sum_{i=1}^{N}2^{-i}\left\{  \left\vert \int_{-T}^{T}\left\langle
w_{1}(t)-w_{2}(t),y_{i}(t)\right\rangle \,dt\right\vert \right. \nonumber\\
&  \qquad\qquad+\left.  \left\vert \int_{\mathbb{R}\setminus\left[
-T,T\right]  }\left\langle w_{1}(t)-w_{2}(t),y_{i}(t)\right\rangle
\,dt\right\vert \right\}  +\frac{\varepsilon}{2}\nonumber\\
&  <\max_{i=1,\dots,N}\int_{-T}^{T}\left\vert w_{1}(t)-w_{2}(t)\right\vert
\;\,\left\vert y_{i}(t)\right\vert \,dt+\varepsilon.\nonumber
\end{align}
Hence it suffices to show that for all considered pairs of control functions
the integrands vanish. This is immediate from the definition of $v_{j}%
,\,j=-2,\dots,n+1$.

(ii) Let $\mathcal{E}$ be a maximal invariant chain transitive set in
$\mathcal{U}\times\Omega\times M$. For $(\omega,x)\in\pi_{\Omega\times
M}\mathcal{E}$ there exists $u\in\mathcal{U}$ such that $(\omega\cdot
t,\varphi(t,\omega,x,u))\in\pi_{\Omega\times M}\mathcal{E}$ for all
$t\in\mathbb{R}$. Now let $(\omega,x),(\overline{\omega},y)\in\pi
_{\Omega\times M}\mathcal{E}$ and fix $\varepsilon,T>0$. There are
$u,v\in\mathcal{U}$ with $(u,\omega,x),(v,\overline{\omega},y)\in\mathcal{E}$.
Then, by chain transitivity of $\mathcal{E}$, there exists an $(\varepsilon
,T)$-chain from $(u,\omega,x)$ to $(v,\overline{\omega},y)$. This yields a
controlled $(\varepsilon,T)$-chain from $(\omega,x)$ to $(\overline{\omega
},y)$.

The proof of the theorem is concluded by the observation that $E$ is maximal
if and only if $\mathcal{E}$ is maximal.
\end{proof}

Observe that, under the compactness assumption of Proposition
\ref{Proposition_flow_inv}, the maximal chain transitive sets of the control
flow $\Phi$ are invariant, and hence the lifts of the chain control sets
coincide with the maximal chain transitive sets for $\Phi$.

\section{Control sets\label{NONAUTONOMOUS-CS}}

This section introduces nonautonomous control sets. Nonautonomous equilibria
for the uncontrolled system are contained in control sets, which are related
to topologically mixing sets of the control flow.

\begin{definition}
\label{Definition_control_set}A nonvoid set $D\subset\Omega\times M$ is a
(nonautonomous) control set if it has the following properties:

(i) for all $(\omega,x_{0})\in D$ there is a control $u$ such that%
\[
\psi(t,\omega,x_{0},u)=(\omega\cdot t,\varphi(t,\omega,x_{0},u))\in D\text{
for all }t\geq0;
\]

(ii) for all $(\omega,x_{0})\in D$ the closure of the (extended) reachable set
from $(\omega,x_{0})$,%
\[
\mathbf{R}^{e}(\omega,x_{0}):=\left\{  \psi(t,\omega,x_{0},u)\in\Omega\times
M\left\vert t\geq0\text{ and }u\in\mathcal{U}\right.  \right\}
\]
contains $D$, i.e., $D\subset\mathrm{cl}\mathbf{R}^{e}(\omega,x_{0})$ for all
$(\omega,x_{0})\in D$, and

(iii) $D$ is maximal with these properties.
\end{definition}

Recall that $\tau$ denotes the flow of the uncontrolled system; cf. (\ref{tau}).

\begin{proposition}
\label{Proposition_minimal1}Let $\mathcal{K}\subset\Omega\times M$ be a
minimal $\tau$-invariant set. Then there exists a control set $D$ with
$\mathcal{K}\subset D$.
\end{proposition}

\begin{proof}
First, we observe that any set $D^{0}$ satisfying properties (i) and (ii) of
control sets is contained in a maximal set with these properties, i.e., a
control set. This follows since the union $D$ of all sets $D^{\prime}$
containing $D^{0}$ and satisfying these properties again satisfies property
(i). For property (ii) let $(\omega^{1},x^{1}),(\omega^{2},x^{2})\in D$. Then
there is $(\omega^{3},x^{3})\in D^{0}$ with $(\omega^{3},x^{3})\in
\mathrm{cl}\mathbf{R}^{e}(\omega^{1},x^{1})$ and $(\omega^{2},x^{2}%
)\in\mathrm{cl}\mathbf{R}^{e}(\omega^{3},x^{3})$. Using continuity of $\psi$
with respect to the initial value one shows that $(\omega^{2},x^{2}%
)\in\mathrm{cl}\mathbf{R}^{e}(\omega^{1},x^{1})$. Certainly $D$ is maximal
with properties (i) and (ii), and hence a control set.

Since the set $\mathcal{K}$ is $\tau$-invariant it satisfies condition (i) by
choosing the control $u=0$. Condition (ii) holds since, for every $(\omega
_{0},x_{0})\in\mathcal{K}$, the limit set%
\[
\left\{  (\omega,x)\in\Omega\times M\left\vert \exists t_{k}\rightarrow
\infty:\psi(t_{k},\omega_{0},x_{0},0)\rightarrow(\omega,x)\}\right.  \right\}
\subset\mathrm{cl}\mathbf{R}^{e}(\omega_{0},x_{0})
\]
is a compact invariant set contained in $\mathcal{K}$ and hence coincides with
$\mathcal{K}$ by minimality. Thus, it follows that $\mathcal{K}\subset
\mathrm{cl}\mathbf{R}^{e}(\omega_{0},x_{0})$ showing that $\mathcal{K}$ is
contained in a control set.
\end{proof}

The system cannot leave a control set and return to it.

\begin{proposition}
\label{Proposition_noreturn}Let $D$ be a control set and assume that there are
$(\omega_{0},x_{0})\in D$, a time $t_{0}>0$, and a control $u_{0}%
\in\mathcal{U}$ such that $\psi(t_{0},\omega_{0},x_{0},u_{0})\in D$. Then it
follows that $\psi(t_{1},\omega_{0},x_{0},u_{0})\in D$ for all $t_{1}%
\in\lbrack0,t_{0}]$.
\end{proposition}

\begin{proof}
Let $t_{1}\in\lbrack0,t_{0}]$. Since $\psi(t_{1},\omega_{0},x_{0},u_{0}%
)\in\mathrm{cl}\mathbf{R}^{e}(\omega_{0},x_{0})$ continuity of $\psi$ implies
that $\psi(t_{1},\omega_{0},x_{0},u_{0})\in\mathrm{cl}\mathbf{R}^{e}%
(\omega,x)$ for all $(\omega,x)\in D$. Since $D\subset\mathrm{cl}%
\mathbf{R}^{e}(\psi(t_{0},\omega_{0},x_{0},u_{0}))$ and%
\[
\psi(t_{0},\omega_{0},x_{0},u_{0})=\psi(t_{0}-t_{1},\psi(t_{1},\omega
_{0},x_{0},u_{0}),u_{0}(t_{1}+\cdot))
\]
it follows that $D\subset\mathrm{cl}\mathbf{R}^{e}\left(  \psi(t_{1}%
,\omega_{0},x_{0},u_{0})\right)  $. This proves property (ii) of control sets.
Property (i) follows by the maximality property of control sets since%
\[
D\cup\left\{  \psi(t,\omega_{0},x_{0},u_{0})\left\vert t\in\lbrack
0,t_{0}]\right.  \right\}
\]
satisfies properties (i) and (ii) of control sets.
\end{proof}

In terms of the fibers $D_{\omega}$ of a control set $D$, the assumption of
Proposition \ref{Proposition_noreturn} may be written as $x_{0}\in
D_{\omega_{0}}$ and $\varphi(t_{0},\omega_{0},x_{0},u_{0})\in D_{\omega
_{0}\cdot t_{0}}$.

Next we concentrate on controllability properties of the component in $M$. For
$(\omega,x)\in\Omega\times M$, define the reachable and controllable sets at
time $T>0$ by%
\begin{align*}
\mathbf{R}_{T}(\omega,x)  &  :=\left\{  \varphi(T,\omega,x,u)\left\vert
u\in\mathcal{U}\right.  \right\}  ,\\
\mathbf{C}_{T}(\omega,x)  &  :=\left\{  y\left\vert \exists u\in
\mathcal{U}:x=\varphi(T,\omega\cdot(-T),y,u)\right.  \right\}  ,
\end{align*}
respectively. We will consider generalized equilibria of the uncontrolled
system with flow $\tau$ given by (\ref{tau}).

\begin{definition}
A map $\alpha:\Omega\rightarrow M$ is a $\tau$-equilibrium if $\alpha
(\omega\cdot t)=\varphi(t,\omega,\alpha(\omega),0)$ for all $t\in\mathbb{R}$
and $\omega\in\Omega$.
\end{definition}

For a $\tau$-equilibrium $\alpha$, the graph $\mathrm{gr}(\alpha)=\left\{
(\omega,\alpha(\omega))\in\Omega\times M\left\vert \omega\in\Omega\right.
\right\}  $ is an invariant set for the flow $\tau$ since%
\[
\tau(t,\omega,\alpha(\omega))=(\omega\cdot t,\varphi(t,\omega,\alpha
(\omega),0))=(\omega\cdot t,\alpha(\omega\cdot t))\text{ for all }%
t\in\mathbb{R}.
\]
When $\alpha$ is continuous, the image $\alpha(\Omega)$ and the graph
$\mathrm{gr}(\alpha)$ are compact. In particular, the graph of $\alpha$ is a
minimal $\tau$-invariant set. In this situation, the graph of $\alpha$ is
called a copy of the base $\Omega$.

For a control set $D$, the interior of a fiber $D_{\omega}$ is%
\[
\mathrm{int}D_{\omega}=\mathrm{int}\left\{  x\in M\left\vert (\omega,x)\in
D\right.  \right\}  .
\]
The next theorem presents a condition which implies that, for a $\tau
$-equilibrium $\alpha$, any point $\alpha(\omega)$ is contained in the
interior of $D_{\omega}$.

\begin{theorem}
\label{Theorem_exact}Let $\alpha$ be a continuous $\tau$-equilibrium. Assume
that there are $\varepsilon,T>0$ such that for every $\omega\in\Omega$
\begin{equation}
\mathbf{B}_{\varepsilon}(\alpha(\omega\cdot T))\subset\mathbf{R}_{T}%
(\omega,\alpha(\omega))\text{ and }\mathbf{B}_{\varepsilon}(\alpha(\omega
\cdot(-T)))\subset\mathbf{C}_{T}(\omega,\alpha(\omega)). \label{CW_4.5}%
\end{equation}
Then there exists a control set $D$ containing the graph $\mathrm{gr}(\alpha)$
and $\alpha(\omega)\in\mathrm{int}D_{\omega}$ for every $\omega\in\Omega$.
\end{theorem}

\begin{proof}
By Proposition \ref{Proposition_minimal1}, there exists a control set $D$
containing the minimal invariant set $\mathrm{gr}(\alpha)$. We will prove that
$\left\{  \omega\right\}  \times\mathbf{B}_{\varepsilon}(\alpha(\omega
))\subset D$ for all $\omega\in\Omega$ showing that $\alpha(\omega
)\in\mathrm{int}D_{\omega}$. For this purpose it suffices to show that
$\bigcup\nolimits_{\omega\in\Omega}\left\{  \omega\right\}  \times
\mathbf{B}_{\varepsilon}(\alpha(\omega))$ satisfies properties (i) and (ii) of
control sets.

Step 1. Let $(\omega_{1},\alpha(\omega_{1})),(\omega_{2},\alpha(\omega
_{2}))\in\mathrm{gr}(\alpha)$. We prove that, for%
\[
y_{1}\in\mathbf{B}_{\varepsilon}(\alpha(\omega_{1})),y_{2}\in\mathbf{B}%
_{\varepsilon}(\alpha(\omega_{2})),
\]
there are $T_{n}\geq0$ and $u_{n}\in\mathcal{U}$ with $\psi(T_{n},\omega
_{2},y_{2},u_{n})\rightarrow(\omega_{1},y_{1})$. This will imply that property
(ii) of control sets holds.

The second part of condition (\ref{CW_4.5}) for $\omega_{1}\cdot T$ implies%
\[
\mathbf{B}_{\varepsilon}(\alpha(\omega_{1}))\subset\mathbf{C}_{T}(\omega
_{1}\cdot T,\alpha(\omega_{1}\cdot T)).
\]
Since $y_{1}\in\mathbf{B}_{\varepsilon}(\alpha(\omega_{1}))$ there exists
$v_{1}\in\mathcal{U}$ with%
\begin{equation}
(\omega_{1}\cdot T,\alpha(\omega_{1}\cdot T))=\psi(T,\omega_{1},y_{1}%
,v_{1})=(\omega_{1}\cdot T,\varphi(T,\omega_{1},y_{1},v_{1})). \label{step1a}%
\end{equation}
Similarly, the first part of condition (\ref{CW_4.5}) for $\omega_{2}%
\cdot(-T)$ implies%
\[
\mathbf{B}_{\varepsilon}(\alpha(\omega_{2}))\subset\mathbf{R}_{T}(\omega
_{2}\cdot(-T),\alpha(\omega_{2}\cdot(-T))),
\]
and hence for $y_{2}\in\mathbf{B}_{\varepsilon}(\alpha(\omega_{2}))$ there
exists a control $v_{2}\in\mathcal{U}$ with%
\[
(\omega_{2},y_{2})=\psi(T,\omega_{2}\cdot(-T),\alpha(\omega_{2}(-T),v_{2}).
\]
Since $\mathrm{gr}(\alpha)$ is a minimal $\tau$-invariant set, there are
$S_{n}\rightarrow\infty$ with%
\begin{align}
\psi(S_{n},\omega_{1}\cdot T,\alpha(\omega_{1}\cdot T),0)  &  =\tau
(S_{n},\omega_{1}\cdot T,\alpha(\omega_{1}\cdot T))\nonumber\\
&  =(\omega_{1}\cdot(S_{n}+T),\varphi(S_{n}+T,\omega_{1},\alpha(\omega
_{1}),0)\label{step1bb}\\
&  \rightarrow(\omega_{2}\cdot(-T),\alpha(\omega_{2}\cdot(-T))=\psi
(-T,\omega_{2},\alpha(\omega_{2}),0).\nonumber
\end{align}
By continuity of $\psi$, this implies%
\begin{align*}
&  \psi(T,\psi(S_{n},\omega_{1}\cdot T,\alpha(\omega_{1}\cdot T),0),v_{2})\\
&  \rightarrow\psi(T,\psi(-T,\omega_{2},\alpha(\omega_{2}),0),v_{2}%
)=\psi(T,\omega_{2}\cdot(-T),\alpha(\omega_{2}\cdot(-T)),v_{2})=(\omega
_{2},y_{2}).
\end{align*}
Define the concatenated controls%
\[
u_{n}(t)=\left\{
\begin{array}
[c]{lll}%
v_{1}(t) & \text{for} & t\in\lbrack0,T]\\
0 & \text{for} & t\in(T,S_{n}+T]\\
v_{2}(t-S_{n}-T) & \text{for} & t\in(S_{n}+T,S_{n}+2T]
\end{array}
\right.  .
\]
Then, with $T_{n}:=S_{n}+2T$, it follows that%
\[
\psi(T_{n},\omega_{1},y_{1},u_{n})=\psi(T,\psi(S_{n},\psi(T,\omega_{1}%
,y_{1},v_{1}),0),v_{2})\rightarrow(\omega_{2},y_{2}).
\]
This shows that all $(\omega,y)\in\Omega\times\mathbf{B}_{\varepsilon}%
(\alpha(\omega))$ satisfy property (ii) of control sets.

Step 2. Concerning property (i) of control sets, let $(\omega,y)\in
\Omega\times\mathbf{B}_{\varepsilon}(\alpha(\omega))$. As shown above, there
are $S_{1}^{\prime}:=T_{1}\geq2T$ and $u_{1}:=v_{1}\in\mathcal{U}$ with
$\psi(S_{1}^{\prime},\omega,y,u_{1})=\psi(T_{1},\omega,y,v_{1})\in\Omega
\times\mathbf{B}_{\varepsilon}(\alpha(\omega\cdot S_{1}))$. By Proposition
\ref{Proposition_noreturn}, it follows that all points $\psi(t,\omega
,y,u_{1}),t\in\lbrack0,S_{1}^{\prime}],$ are in $D$. Repeating this argument
one finds a time $S_{2}^{\prime}\geq2T$ and a control $u_{2}\in\mathcal{U}$
such that%
\[
\psi(S_{2}^{\prime},\psi(S_{1}^{\prime},\omega,y,u_{1}),u_{2})\in\Omega
\times\mathbf{B}_{\varepsilon}(\alpha(\omega\cdot(S_{1}^{\prime}+S_{2}%
^{\prime}))).
\]
Proceeding in this way, one constructs a control keeping the system in $D$ for
all $t\geq0$.

Steps 1 and 2 show the assertion of the theorem.
\end{proof}

Finally, we relate control sets around nonautonomous equilibria to
topologically mixing sets of the control flow. Recall that a flow $(X,\phi)$
on a metric space $X$ is topologically mixing if for any two open sets
$\emptyset\not =V_{1},V_{2}\subset X$ there is $S>0$ with $\phi(-S,V_{1})\cap
V_{2}\not =\emptyset$. In the autonomous case, the lifts of control sets with
nonvoid interior to $\mathcal{U}\times M$ are the maximal topologically mixing
sets of the control flow; cf. Colonius and Kliemann \cite[Theorem
4.3.8]{ColK00}. In the following theorem, we assume a strengthened version of
condition (\ref{CW_4.5}).

\begin{theorem}
\label{Theorem_equivalence2}Let $\alpha$ be a continuous $\tau$-equilibrium.
Assume that there are $\varepsilon\geq\varepsilon_{0}>0$ and $T>0$ such that
for every $\omega\in\Omega$ one has $\mathbf{B}_{\varepsilon}(\alpha
(\omega\cdot T))\subset\mathbf{R}_{T}(\omega,\alpha(\omega))$ and
\begin{equation}
d((\omega^{\prime},y^{\prime}),(\omega,\alpha(\omega))<\varepsilon_{0}\text{
implies }\mathbf{B}_{\varepsilon}(\alpha(\omega\cdot(-T)))\subset
\mathbf{C}_{T}(\omega^{\prime},y^{\prime}). \label{5.21}%
\end{equation}

(i) Then, for all $\omega_{1},\omega_{2}\in\Omega$, all $y_{1}\in
\mathbf{B}_{\varepsilon}(\alpha(\omega_{1}))$, and all $y_{2}\in
B_{\varepsilon_{0}}(\alpha(\omega_{2}))$, there are $T_{n}\rightarrow\infty$
and $u_{n}\in\mathcal{U}$ such that $\varphi(T_{n},\omega_{1},y_{1}%
,u_{n})=y_{2}$ for all $n\in\mathbb{N}$ and $\omega_{1}\cdot T_{n}%
\rightarrow\omega_{2}$ for $n\rightarrow\infty$.

(ii) If in the assumption above $T$ can be chosen large enough, it follows for
the control set $D$ containing the graph $\mathrm{gr}(\alpha)$ that the set%
\[
\mathcal{D}^{\prime}:=\left\{  (u,\omega,x)\in\mathcal{U}\times D\left\vert
d(\varphi(t,\omega,x,u),\alpha(\omega\cdot t))<\varepsilon_{0}\text{ for all
}t\in\mathbb{R}\right.  \right\}
\]
is a topologically mixing set for the control flow $\Phi$.
\end{theorem}

\begin{proof}
(i) Let $(\omega_{1},\alpha(\omega_{1})),(\omega_{2},\alpha(\omega_{2}%
))\in\mathrm{gr}(\alpha)$. As shown in (\ref{step1a}), for $y_{1}\in
\mathbf{B}_{\varepsilon}(\alpha(\omega_{1}))$ there exists $v_{1}%
\in\mathcal{U}$ with%
\[
\psi(T,\omega_{1},y_{1},v_{1})=(\omega_{1}\cdot T,\varphi(T,\omega_{1}%
,y_{1},v_{1}))=(\omega_{1}\cdot T,\alpha(\omega_{1}\cdot T))
\]
and by (\ref{step1bb}) there are $S_{n}\rightarrow\infty$ with%
\[
\psi(S_{n},\omega_{1}\cdot T,\alpha(\omega_{1}\cdot T),0)\rightarrow
\psi(-T,\omega_{2},\alpha(\omega_{2}),0)=(\omega_{2}\cdot(-T),\alpha
(\omega_{2}\cdot(-T))).
\]
For $n$ large enough, this implies that%
\[
\varphi(S_{n},\omega_{1}\cdot T,\alpha(\omega_{1}\cdot T),0)\in\mathbf{B}%
_{\varepsilon}(\alpha(\omega_{2}\cdot(-T))).
\]
There is $\delta>0$ such that
\[
d(\omega_{1}\cdot(S_{n}+T),\omega_{2}\cdot(-T))<\delta\text{ implies }%
d(\omega_{1}\cdot(S_{n}+2T),\omega_{2})<\varepsilon_{0}.
\]
For $\omega^{\prime}:=\omega_{1}\cdot(S_{n}+2T)$ and $y^{\prime}=y_{2}$, it
holds that $d((\omega^{\prime},y^{\prime}),(\omega_{2},\alpha(\omega
_{2}))<\varepsilon_{0}$. By (\ref{5.21})%
\[
\mathbf{B}_{\varepsilon}(\alpha(\omega_{2}\cdot(-T)))\subset\mathbf{C}%
_{T}(\omega^{\prime},y^{\prime}),
\]
and hence it follows that there exists $v_{2}\in\mathcal{U}$ with%
\[
\varphi(T,\omega_{1}\cdot(S_{n}+2T),\varphi(S_{n},\omega_{1}\cdot
T,\alpha(\omega_{1}\cdot T),0),v_{2})=y_{2}.
\]
Define a control $u_{n}\in\mathcal{U}$ by%
\[
u_{n}(t)=\left\{
\begin{array}
[c]{lll}%
v_{1}(t) & \text{for} & t\in\lbrack0,T]\\
0 & \text{for} & t\in(T,S_{n}+T]\\
v_{2}(t-S_{n}-T) & \text{for} & t\in(S_{n}+T,S_{n}+2T]
\end{array}
\right.  .
\]
Then, with $T_{n}:=S_{n}+2T$ it follows that%
\[
\psi(T_{n},\omega_{1},y_{1},u_{n})=\psi(T,\psi(S_{n},\psi(T,\omega_{1}%
,y_{1},v_{1}),0),v_{2})=(\omega_{1}\cdot T_{n},y_{2}).
\]
(ii) Let $\emptyset\not =V_{1}^{\prime},V_{2}^{\prime}\subset\mathcal{D}%
^{\prime}$ be open. We have to show that there are $S>0$ and $(u,\omega,x)\in
V_{1}^{\prime}$ with $\Phi(-S,u,\omega,x)\in V_{2}^{\prime}$. The sets
$V_{j}^{\prime},j=1,2$, have the form $V_{j}^{\prime}=V_{j}\cap\mathcal{D}%
^{\prime}$, where $V_{j}$ are open subsets of $\mathcal{U}\times\Omega\times
M$. Using a base of the weak$^{\ast}$ topology on $\mathcal{U}$ (cf. Kawan
\cite[p. 20]{Kawan13}) we may further assume that for some $(v_{j},\omega
_{j},x_{j})\in\mathcal{U}\times\Omega\times M$ with $d(\varphi(t,\omega
_{j},x_{j},v_{j}),\alpha(\omega_{j}\cdot t))<\varepsilon_{0}$ for all
$t\in\mathbb{R}$, one has
\[
V_{j}=W(v_{j})\times\mathbf{B}_{\delta}(\omega_{j})\times\mathbf{B}_{\delta
}(x_{j}),j=1,2,
\]
where $\delta>0,k_{j}\in\mathbb{N}$, and
\begin{align*}
W(v_{j})  &  =\left\{  u\in\mathcal{U}\mid\left\vert \int_{\mathbb{R}%
}\left\langle v_{j}(\tau)-u(t),y_{ij}(\tau)\right\rangle d\tau\right\vert
<\delta\text{ for }i=1,\dots,k_{j}\right\}  ,\\
\mathbf{B}_{\delta}(\omega_{j})  &  =\left\{  \omega\in\Omega\left\vert
d(\omega_{j},\omega)<\delta\right.  \right\}  ,\mathbf{B}_{\delta}%
(x_{j})=\left\{  x\in M\left\vert d(x_{j},x)<\delta\right.  \right\}  .
\end{align*}
There is $T_{1}>0$ such that, for $j=1,2$ and $i=1,\dots,k_{j}$,%
\[
\int_{\mathbb{R}\setminus\lbrack-T_{1},T_{1}]}\left\vert y_{ij}(t)\right\vert
dt<\frac{\varepsilon}{\mathrm{diam}U}\text{ with }\mathrm{diam}U=\max_{u,v\in
U}\left\Vert u-v\right\Vert .
\]
By assumption, we may take $T\geq T_{1}$. Since $\varphi(T,\omega_{2}%
,x_{2},v_{2})\in\mathbf{B}_{\varepsilon}(\alpha(\omega_{2}\cdot T))$ and
$\varphi(-T,\omega_{1},x_{1},v_{1})\in\mathbf{B}_{\varepsilon}(\alpha
(\omega_{1}\cdot(-T))$, there are $S_{n}\rightarrow\infty$ and $v_{n}%
\in\mathcal{U}$ such that%
\[
\varphi(S_{n},\psi(T,\omega_{2},x_{2},v_{2}),v_{n})=\varphi(-T,\omega
_{1},x_{1},v_{1})\text{ and }\omega_{2}\cdot(T+S_{n})\rightarrow\omega
_{1}\cdot(-T)\text{ for }n\rightarrow\infty.
\]
Continuity of $\psi$ implies%
\[
\psi(T,\psi(S_{n},\psi(T,\omega_{2},x_{2},v_{2}),v_{n}),v_{1}(-T+\cdot
))\rightarrow\psi(T,\psi(-T,\omega_{1},x_{1},v_{1}),v_{1}(-T+\cdot
))=(\omega_{1},x_{1}).
\]
It follows that for $n>2$ large enough%
\[
(\omega_{0},z_{0}):=\psi(T,\psi(S_{n},\psi(T,\omega_{2},x_{2},v_{2}%
),v_{n}),v_{1}(-T+\cdot))\in\mathbf{B}_{\delta}(\omega_{1})\times
\mathbf{B}_{\delta}(x_{1}).
\]
Define a control $u\in\mathcal{U}$ by%
\[
u(t)=\left\{
\begin{array}
[c]{ccc}%
v_{2}(t) & \text{for} & t\in(-\infty,T]\\
v_{n}(t-T) & \text{for} & t\in(T,T+S_{n}]\\
v_{1}(t-S_{n}-2T) & \text{for} & t\in(T+S_{n},\infty)
\end{array}
\right.  .
\]
We find for $i=1,\dots,k_{2}$%
\begin{align*}
&  \left\vert \int_{\mathbb{R}}\left\langle v_{2}(t)-u(t),y_{i2}%
(t)\right\rangle dt\right\vert \\
&  \leq\left\vert \int_{-T}^{T}\left\langle v_{2}(t)-u(t),y_{i2}%
(t)\right\rangle dt\right\vert +\left\vert \int_{\mathbb{R}\setminus
\lbrack-T,T]}\left\langle v_{2}(t)-u(t),y_{i2}(t)\right\rangle dt\right\vert
\\
&  \leq0+\mathrm{diam}U\cdot\int_{\mathbb{R}\setminus\lbrack-T_{1},T_{1}%
]}\left\vert y_{i2}(t)\right\vert dt<\varepsilon.
\end{align*}
This proves that $u\in W(v_{2})$ and similarly it follows that $u(S_{n}%
+2T+\cdot)\in W(v_{1})$. Furthermore, by construction one has that, with
$S:=S_{n}+2T$,%
\begin{align*}
\omega_{0}  &  =\omega_{2}\cdot(T+S_{n}+T)=\omega_{2}\cdot S,\\
z_{0}  &  =\varphi(T,\psi(S_{n},\psi(T,\omega_{2},x_{2},v_{2}),v_{n}%
),v_{1}(-T+\cdot))\\
&  =\varphi(S_{n}+2T,\omega_{2},x_{2},u)=\varphi(S,\omega_{2},x_{2},u).
\end{align*}
This implies that $\Phi(-S,V_{1})\cap V_{2}\not =\emptyset$ since%
\begin{align*}
(u(S+\cdot),\omega_{0},z_{0})  &  \in W(v_{1})\times\mathbf{B}_{\delta}%
(\omega_{1})\times\mathbf{B}_{\delta}(x_{1})=V_{1},\\
\Phi(-S,u(S+\cdot),\omega_{0},z_{0})  &  =(u,\omega_{2},x_{2})\in
W(v_{2})\times\mathbf{B}_{\delta}(\omega_{2})\times\mathbf{B}_{\delta}%
(x_{2})=V_{2}.
\end{align*}

\end{proof}

\begin{remark}
\label{Remark_hull}Concerning the scalar Example \ref{Example_hull}, Elia,
Fabbri, and N\'{u}\~{n}ez \cite[Theorem 3.4, Theorem 3.6, and Theorem
3.8]{EFN25} present several sufficient conditions for the existence of
continuous equilibria $\alpha$. If one chooses the control range $U=[\rho
_{1},\rho_{2}]$ large enough, one easily sees that the assumptions of Theorem
\ref{Theorem_exact} and Theorem \ref{Theorem_equivalence2} can be satisfied.
\end{remark}

\end{document}